\documentclass[10pt]{amsart}
\usepackage{geometry}                
\usepackage[parfill]{parskip}    
\usepackage{breqn}
\usepackage{amssymb, amsthm, amsmath}
\usepackage{graphicx}
\usepackage{xypic}
\usepackage{tikz}
\usepackage{color}
\usepackage{dbnsymb}
\usepackage{hyperref}
\usepackage{enumitem}
\usepackage{subcaption}

\makeatletter
\@namedef{subjclassname@1991}{2020 Mathematics Subject Classification}
\makeatother

\newtheorem{theorem}{Theorem}[section]
\newtheorem{lemma}[theorem]{Lemma}
\newtheorem{proposition}[theorem]{Proposition}

\newtheorem{corollary}[theorem]{Corollary}
\newtheorem{question}[theorem]{Question}

\theoremstyle{definition}
\newtheorem{definition}[theorem]{Definition}
\newtheorem{remark}[theorem]{Remark}
\newtheorem{example}[theorem]{Example}

\newcommand{\Z}{\mathbb{Z}}

\newcommand{\hb}{\langle \hsmoothing \rangle}
\newcommand{\vb}{\langle \smoothing \rangle}

\DeclareMathOperator{\spn}{span}
\DeclareMathOperator{\tri}{tri}
\DeclareMathOperator{\diam}{diam}

\title{ Quotients of the Gordian and H(2)-Gordian graphs}

\author{Christopher Flippen}
\author{Allison H. Moore}
\author{Essak Seddiq}

\address{Christopher Flippen\\ Department of Mathematics \& Applied Mathematics \\ 
   Virginia Commonwealth University \\Richmond, VA 23284 \\ USA  
   }
\address{Allison H. Moore\\ Department of Mathematics \& Applied Mathematics \\ 
   Virginia Commonwealth University \\Richmond, VA 23284 \\ USA  
   }
\address{Essak Seddiq\\ Department of Mathematics \& Applied Mathematics \\ 
   Virginia Commonwealth University \\Richmond, VA 23284 \\ USA  
   }


\begin{document}

\begin{abstract}
The Gordian graph and H(2)-Gordian graphs of knots are abstract graphs whose vertex sets represent isotopy classes of unoriented knots, and whose edge sets record whether pairs of knots are related by crossing changes or H(2)-moves, respectively. We investigate quotients of these graphs under equivalence relations defined by several knot invariants including the determinant, the span of the Jones polynomial, and an invariant related to tricolorability. We show, in all cases considered, that the quotient graphs are Gromov hyperbolic. We then prove a collection of results about the graph isomorphism type of the quotient graphs. In particular, we find that the H(2)-Gordian graph of links modulo the relation induced by the span of the Jones polynomial is isomorphic with the complete graph on infinitely many vertices. 
\end{abstract}
\maketitle

\section{Introduction}

The Gordian graph is an abstract graph that organizes the set of knots related by crossing changes. More specifically, the vertex set of the Gordian graph $\mathcal{K}_{\backoverslash}$ corresponds with isotopy classes of knots in the three-sphere. A pair of vertices $[K], [J]$ determines an edge whenever knots $K$ and $J$ are related by a crossing change. The Gordian graph is immensely complex. It is a countably infinite graph of infinite diameter with infinite valence at every vertex. Despite the abundance of edges in this graph, measuring the Gordian distance between an arbitrary pair of vertices remains an intractable problem. By Gordian distance, we mean the minimal number of crossing changes required to transform a knot $K$ into a knot $J$ in any sequence of knot diagrams. Gordian distance generalizes the unknotting number of a knot, a well-known knot invariant that is easily defined but difficult to calculate.

In  \cite{JLM}, the authors defined and studied a general class of \emph{knot graphs}, which includes both the Gordian graph, the H(2)-Gordian graph (see section \ref{subsec:knotgraphs} for a definition) and other graphs constructed with different local unknotting operations. The class of knot graphs also includes quotients of the Gordian graph and its relatives under equivalence relations defined by knot invariants. Although the global structure of such knot graphs is not well-understood, there are a few sporadic results. Hirasawa and Uchida showed that every vertex is contained in the complete graph on countably many vertices \cite{HirasawaUchida}.  Zhang, Yang and Lei adapted this result to the H($n$)-Gordian graph \cite{ZhangYang2018, ZhangYang}. Gambaudo and Ghys constructed a quasi-isometric embedding of a Euclidean lattice in the Gordian graph \cite{GambaudoGhys}, and Jabuka, Liu and Moore proved that knot graphs in general fail to be Gromov hyperbolic \cite{JLM}. Other articles that have considered the structure and quotients of knot graphs include \cite{Baader1, Blair, IchiharaJong, NakanishiOhyama, Yoshiyuki}, to name a few. 

In this article, we further investigate quotients of the Gordian graph and H(2)-Gordian graphs under a collection of knot invariants, all of which can be derived from the Jones polynomial. In particular, we will consider the Gordian graph $\mathcal{K}_{\backoverslash}$ and H(2)-Gordian graph $\mathcal{K}_{\smoothing}$, and their quotients under the equivalence relations defined by the span of the Jones polynomial, the determinant, and an invariant $\beta$ that is related to the tricoloring number of the knot. The first question we consider is whether these quotient graphs are hyperbolic. Recall that in a metric graph, a geodesic triangle is $\delta$-thin if each edge of the triangle is contained in a closed $\delta$-neighborhood of the union of the remaining two. A graph is called $\delta$-hyperbolic (or Gromov hyperbolic) if every geodesic triangle is $\delta$-thin for some parameter $\delta\geq0$. (See section \ref{sec:hyperbolic background} for details.) We show:

\begin{theorem}
\label{main hyperbolicity}
The quotients of the Gordian and H(2)-Gordian graphs by the span of the Jones polynomial, the determinant, and the invariant $\beta$ are all Gromov hyperbolic.
\end{theorem}

After establishing the hyperbolicity of these graphs, we turn to the question of their graph isomorphism types. We first observe that in the case of the invariant $\beta$, the quotient graphs are quickly seen to be isomorphic with infinite path graphs. Let $P_{\mathbb{N}}$ denote the infinite path graph indexed by positive integers. In particular, the vertex set of $P_{\mathbb{N}}$ corresponds with the set $\{n \mid n>0\}$, and the edge set is the set of pairs $\{ \{n, n+1\} \mid n>0 \}$.  

\begin{theorem}
\label{beta structure intro}
The quotient of the Gordian and H(2)-Gordian graphs with respect to the invariant $\beta$ are both isomorphic to $P_{\mathbb{N}}$. 
\end{theorem}

In \cite[Theorem 1.4]{JLM}, the authors constructed various hyperbolic quotients of the Gordian graph and observed that these contrasted greatly with the usual Gordian and H(2)-Gordian graphs, which are not hyperbolic. All of the hyperbolic quotient graphs constructed by \cite{JLM} are isometric to a subset of $\mathbb{R}$ with the Euclidean metric and in particular, are graphs of infinite diameter. The quotient graphs with respect to $\beta$ in Theorem \ref{beta structure intro} also have this property. However, unlike with $\beta$, the quotient graphs with respect to the determinant and the span of the Jones polynomial are graphs of finite diameter (see the proofs of Theorems \ref{Gamma det} -- \ref{Gamma 2 span}). To our knowledge, these are the first examples of hyperbolic quotient knot graphs that do not arise from the compatibility condition, as defined in \cite{JLM}.

Let $\mathbb{K}_\infty$ denote the complete graph on a countably infinite set of vertices. Given the existence of induced subgraphs isomorphic to $\mathbb{K}_\infty$ at every vertex in the Gordian and H(2)-Gordian graphs \cite{HirasawaUchida, ZhangYang2018, ZhangYang}, it is natural to ask to what extent we may identify complete graphs in their quotients as well. In the case of the determinant we show:

\begin{theorem}
\label{complete subgraph intro}
In the quotient of both the Gordian and H(2)-Gordian graphs by the determinant, the subgraph induced by the set of vertices $\{[n^k], k\geq 0\}$ for all odd $n\geq3$ is isomorphic to $\mathbb{K}_{\infty}$.
\end{theorem}
 
We also observe the existence of numerous other classes of edges in the quotient graphs with respect to the determinant (see Example \ref{edges with det}). In the case of the span of the Jones polynomial, we prove that the quotient of the Gordian graph is nearly isomorphic to $\mathbb{K}_{\infty}$ (see Proposition \ref{Gamma span structure} for a precise statement), but were unable to identify edges of the form $\{[n], [n+1]\}$ for all $n\in\mathbb{N}$. However, when we consider the quotient of the H(2)-Gordian graph \emph{of links} with respect to the span of the Jones polynomial, we find that the graph is indeed isomorphic to $\mathbb{K}_{\infty}$, the complete graph on infinitely many vertices.

\begin{theorem}
\label{Gamma2 span structure intro}
The quotient of the H(2)-Gordian graph of links by the span of the Jones polynomial is isomorphic to $\mathbb{K}_{\infty}$. 
\end{theorem}

A more precise statement of the graph isomorphism types found is given in Section \ref{sec:graphs}. Theorem \ref{complete subgraph intro}, Proposition \ref{Gamma span structure} and \ref{Gamma2 span structure intro} are all proved by constructing edges in these graphs. This is accomplished in many cases with explicit calculations of the span of the Jones polynomial and other invariants for several special classes of knots. 

\textbf{Organization.} In section \ref{sec:knots}, we review knots, introduce quotients of knot graphs, and define the knot invariants that will be relevant to our discussion. Section \ref{sec:bracket} is devoted to the Kauffman bracket method for calculating the Jones polynomial, and we give explicit calculations for specific families of knots. In section \ref{sec:hyperbolic} we prove hyperbolicity for the knot graphs, and in section \ref{sec:graphs} we prove various results about the graph isomorphism type of the knot graphs.

\section{Knots and knot invariants}
\label{sec:knots}

A \emph{knot} is a smooth embedding of a circle in the three-sphere, where two knots are considered equivalent if they agree up to ambient isotopy. A link is a collection of knots, possibly linked together. In this article, we will not make a distinction between an oriented knot and its reverse orientation, but we do distinguish between a knot $K$ and its mirror image $-K$. Links are considered up to orientation, however for the particular quotient graphs that we construct, this choice of orientation will not turn out to matter. 

\subsection{The Gordian and H(2)-Gordian graphs and their quotients}
\label{subsec:knotgraphs}

Any knot can be transformed into the unknot by a finite sequence of crossing changes $\backoverslash \leftrightsquigarrow \slashoverback$. The unknotting number $u(K)$ is the minimum number of crossing changes to unknot a knot, minimized over all sequences of diagrams. The \emph{Gordian distance} $d(K, J)$ is the minimum number of crossing changes required to transform a knot $K$ into a knot $J$. Unknotting number is a basic, fundamental knot invariant, yet it is extremely difficult to calculate in general. An H(2)-move is another important operation on knots and links, originally defined in \cite{HNT}. In a diagram, an H(2)-move is realized by the tangle replacement $\smoothing \leftrightsquigarrow \hsmoothing$, or by the resolution of a crossing as a vertical or horizontal smoothing. As with crossing changes, any knot can be unknotted with a finite sequence of (component-preserving) H(2)-moves. The \emph{H(2)-Gordian distance}, $d_2(K, J)$, is the minimum number of component preserving H(2)-moves required to transform $K$ into $J$ through some finite sequence of knots. In the last section (and in some figures), we will consider H(2)-moves along links that are not component preserving. In particular, the H(2)-distance allowing for H(2)-moves that may change the number of link components (called SH(2)-moves by \cite{HNT}) will occur in this article in the statement of Theorem \ref{Gamma2 span structure intro}.

The following definitions are special cases of \emph{knot graphs}, as defined by Jabuka-Liu-Moore  \cite{JLM}.

\begin{definition} Let $\backoverslash$ denote the crossing change unknotting operation, and let $\smoothing$ denote an H(2)-unknotting operation, which is assumed to be component-preserving unless stated otherwise. Let $p(K)$ be an integer-valued knot invariant. 
\begin{enumerate}[label=(\roman*)]
	\item The Gordian knot graph $\mathcal{K}_{\backoverslash}$ is the graph whose vertices are isotopy classes of unoriented knots, in which a pair of vertices $[K], [K']$ span an edge if there exist a pair of knots $K$, $K'$ that possess diagrams related by a crossing change. Similarly, the H(2)-Gordian knot graph $\mathcal{K}_{\smoothing}$ is the graph whose vertices are isotopy classes of unoriented knots, and in which a pair of vertices $[K], [K']$ span an edge if there exist knots $K, K'$ related by a component-preserving H(2)-move.  
	\item The \emph{quotient graph} $\mathcal{QK}^p_{\backoverslash}$ is a graph whose vertices are equivalence classes $[K]^p_{\backoverslash}$ of knots $K$. A pair of knots $K$ and $L$ are equivalent if $p(K) = p(L)$, and two equivalence classes $[K]^p_{\backoverslash}$ and $[K']^p_{\backoverslash}$ span an edge if there exist knots $L$ and $L'$, equivalent to $K$ and $K'$ respectively, that possess diagrams related by an a crossing change. For H(2)-moves, the quotient graph $\mathcal{QK}^p_{\smoothing}$ is analogously defined. 
\end{enumerate}
\end{definition}

In all of the knot graphs defined above, there exist self-loops at all vertices. These can be realized by crossing changes at nugatory crossings or trivial H(2)-moves. For simplicity, we will ignore such self-loops. For the remainder of the article, we will assume that knot graphs are simple and undirected. 

\subsection{Special families of knots}

\begin{figure}
    \centering
	\begin{subfigure}[t]{.3\textwidth}
		\centering
		    \begin{tikzpicture}
            \node[anchor=south west,inner sep=0] at (0,0) {\includegraphics[width=1in]{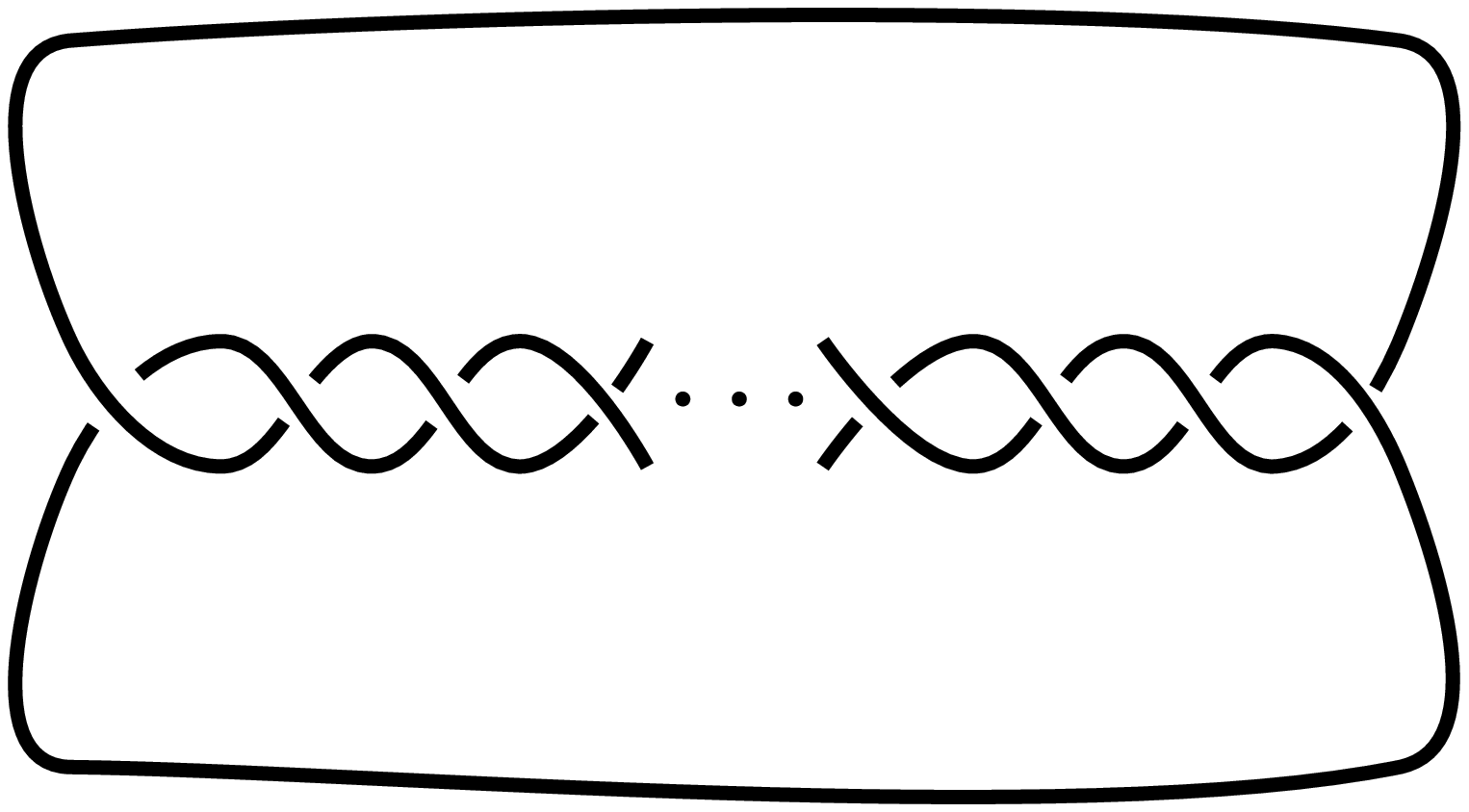}};
            \node[label=above right:{$n$}] at (1,0.75){};
            \end{tikzpicture}
		\caption{$T(2, n)$ torus knots and links}
		\label{fig:torus}
	\end{subfigure}
	\qquad
	\begin{subfigure}[t]{.3\textwidth}
		\centering
			\begin{tikzpicture}
            \node[anchor=south west,inner sep=0] at (0,0) {\includegraphics[width=1in]{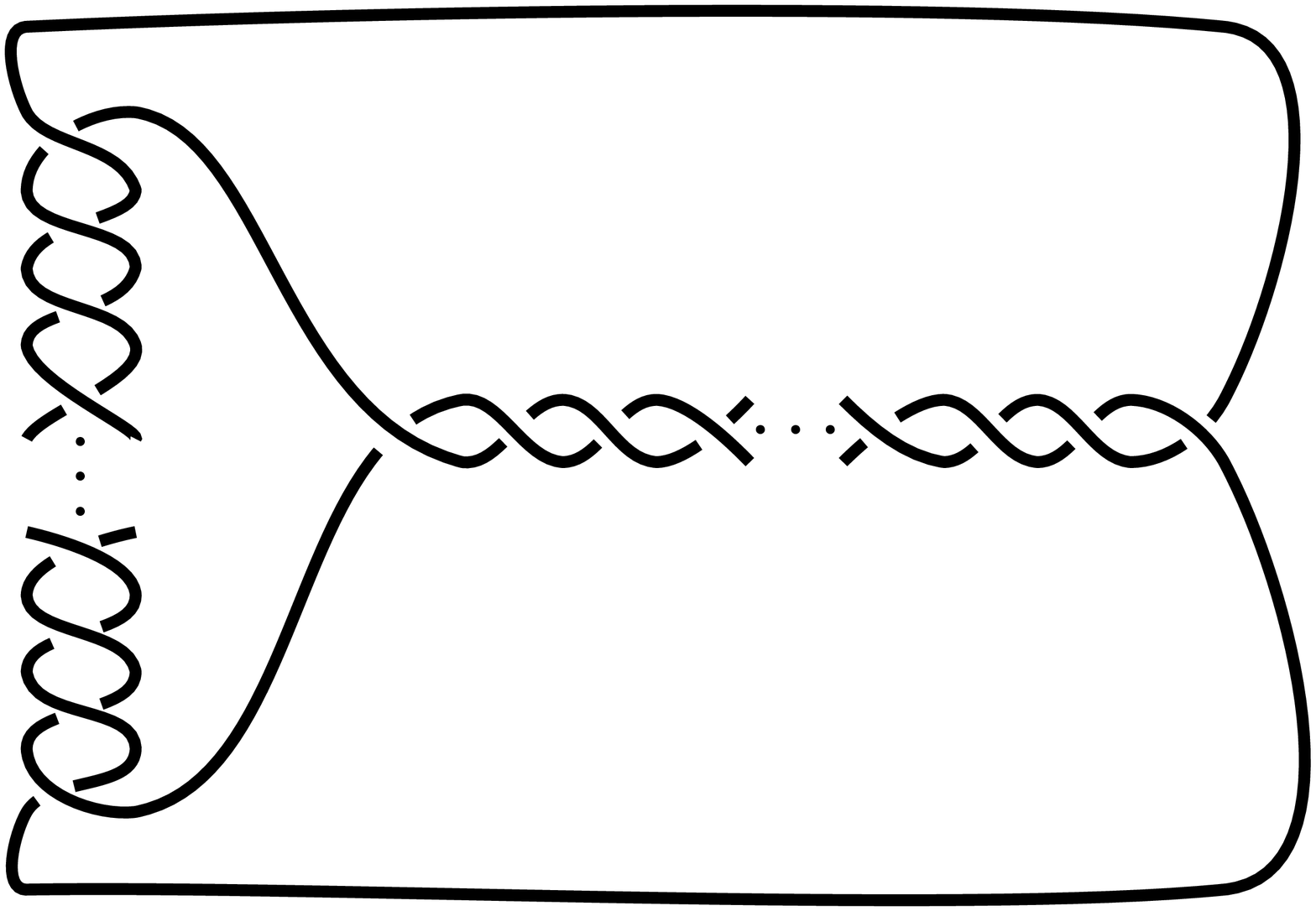}};
            \node[label=above right:{$n$}] at (1,1){};
            \node[label=above right:{$m$}] at (-0.7,1){};
            \end{tikzpicture}
    		\caption{Generalized twist knots $K(m, n)$.}
		\label{fig:twist}
	\end{subfigure}
	\qquad
	\begin{subfigure}[t]{.22\textwidth}
		\centering
			\begin{tikzpicture}
            \node[anchor=south west,inner sep=0] at (0,0) {\includegraphics[width=1in]{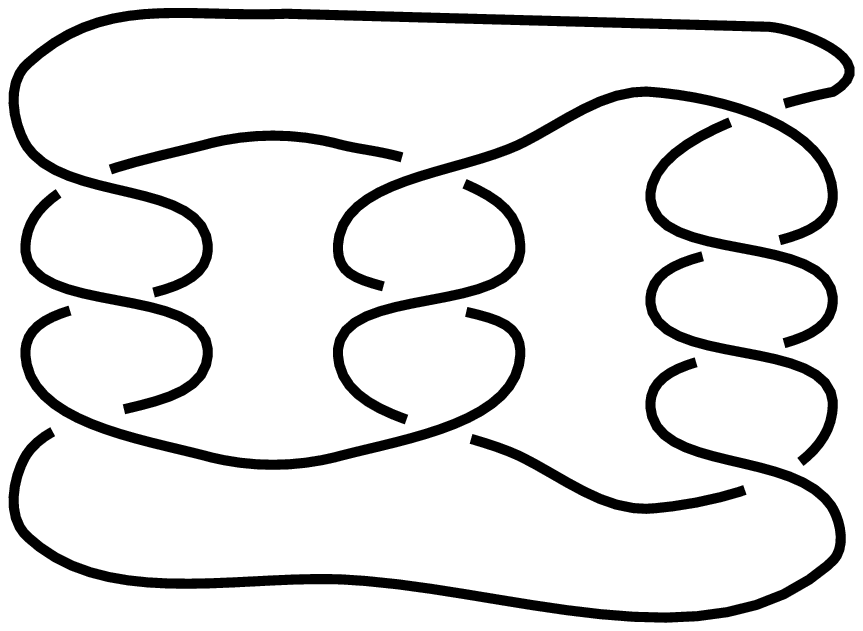}};
            \node[label=above right:{$p$}] at (0.1,0){};
            \node[label=above right:{$q$}] at (1,0){};
            \node[label=above right:{$r$}] at (1.9,-0.1){};
            \end{tikzpicture}
    		\caption{Pretzel knots $K(p, q, r)$.}
		\label{fig:pretzel}
	\end{subfigure}

	\caption{Knot diagrams of $T(2, n)$ torus links, generalized twist knots $K(m,n)$, and $K(p, q, r)$ pretzel knots.}
\label{fig:twist-torus-pretzel}	
\end{figure}

Several standard families of knots and links will be relevant for our arguments. The first of these are the $T(2, n)$ torus links. When $n$ is odd, $T(2, n)$ is a knot and when $n$ is even, $T(2, n)$ is a link of two components. The next are the generalized twist links $K(m,n)$. When either $m$ or $n$ or both $m$ and $n$ are even, $K(m,n)$ is a knot, and when both $m, n$ are odd, $K(m,n)$ is a link of two components. Both $T(2,n)$ torus links and generalized twist links are alternating. A third class of knots that we consider are pretzel knots $P(p, q, r)$. When at most one parameter is even, these are knots, otherwise they are links. See Figure \ref{fig:twist-torus-pretzel}.

It is well known that the twist knots $K(2, n)$ have unknotting number one. More generally, the connected sum of $K(2, n)$ and a knot $J$ is related to $J$ by a crossing change, so that $d(J \# K(2, n), J)=1$ for any knot $J$. Similarly, the $T(2, n)$ torus links are the simplest links that can be related to the unknot by a single H(2)-move, and likewise $d_2(J \# T(2, n), J)=1$ for any knot $J$. There are several other easily found H(2)-moves that relate the $T(2, n)$ torus links and generalized twist links. These moves are shown in Figure \ref{fig:twist-torus-operations}.

In order to prove Theorem \ref{Gamma2 span structure intro}, we will also need to construct two infinite families of knots with specific Jones polynomials. These are the classes of knots $K_q$, where $q\geq 1$ is odd, and $K_{q, r}$, where $q\geq 1$ and $r\geq 3$ are odd, as shown in Figure \ref{fig:specialknots}. These knots can be defined by tangle operations, which we explain in section \ref{sec:bracket}. 

\subsection{Knot invariants}
\label{knot invariants}
We will consider the quotient graphs $\mathcal{QK}^p_{\backoverslash}$ and  $\mathcal{QK}^p_{\smoothing}$ for several knot and link invariants that can be derived from the Jones polynomial. The Jones polynomial $V_L(t)$ is an invariant of oriented links that takes the form of a Laurent polynomial over the integers. If the link has an odd number of components, then $V_L(t)$ is in $\Z[t, t^{-1}]$, and if it has an even number of components, $\sqrt{t} \cdot V_L(t)$ is in $\Z[t, t^{-1}]$ \cite{Jones}. The span of the Jones polynomial, $\spn(L)$, is the difference between the highest and lowest exponents of the Jones polynomial. 

The determinant is an integer valued link invariant which can be obtained by evaluating the Jones polynomial at $t=-1$,
\[
	|V_L(-1)| = \det(L).
\]
For connected sums of knots, $\det(K\# K') = \det(K)\cdot \det(K')$. 
For knots, the determinant is always odd, whereas for links the determinant is even. Another invariant we will consider, $\beta(L)$, is an integer-valued knot invariant that can be obtained both from the tricolorability of the knot and with another evaluation of the Jones polynomial, as described below. Although the Jones polynomial is an invariant of oriented links that is sensitive to mirroring, for the case of knots, none of the invariants of $\spn, \det$, nor $\beta$ distinguish a knot from its reverse orientation nor from its mirror image.

\begin{figure}[t]
	\centering
	\begin{subfigure}[t]{.22\textwidth}
		\centering
		\includegraphics[width=1in]{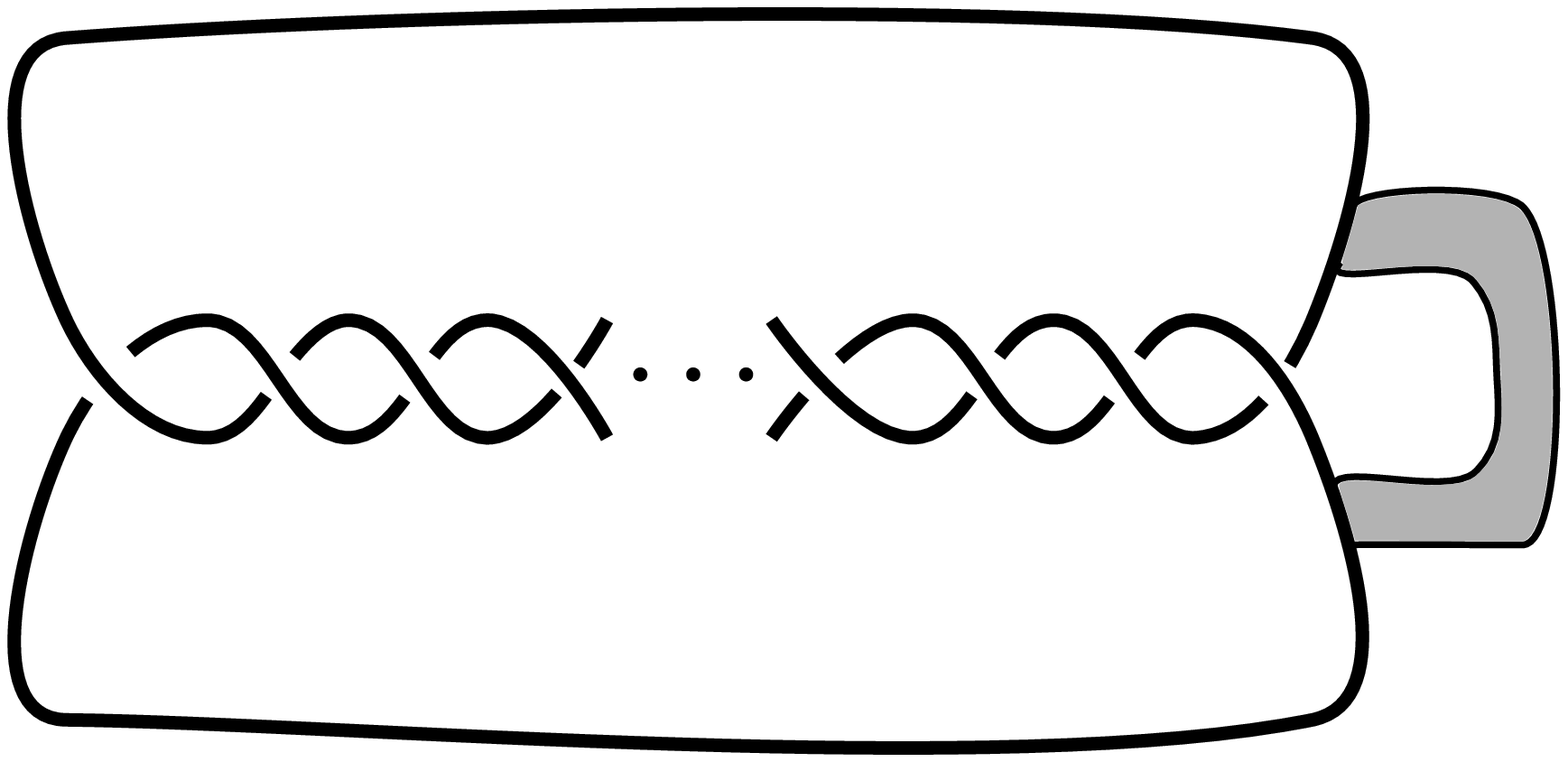}
		\caption{$d_2(T(2,n), U)=1$.}
		\label{fig:torus-operations}
	\end{subfigure}
	\qquad
	\begin{subfigure}[t]{.32\textwidth}
		\centering
		\includegraphics[width=1in]{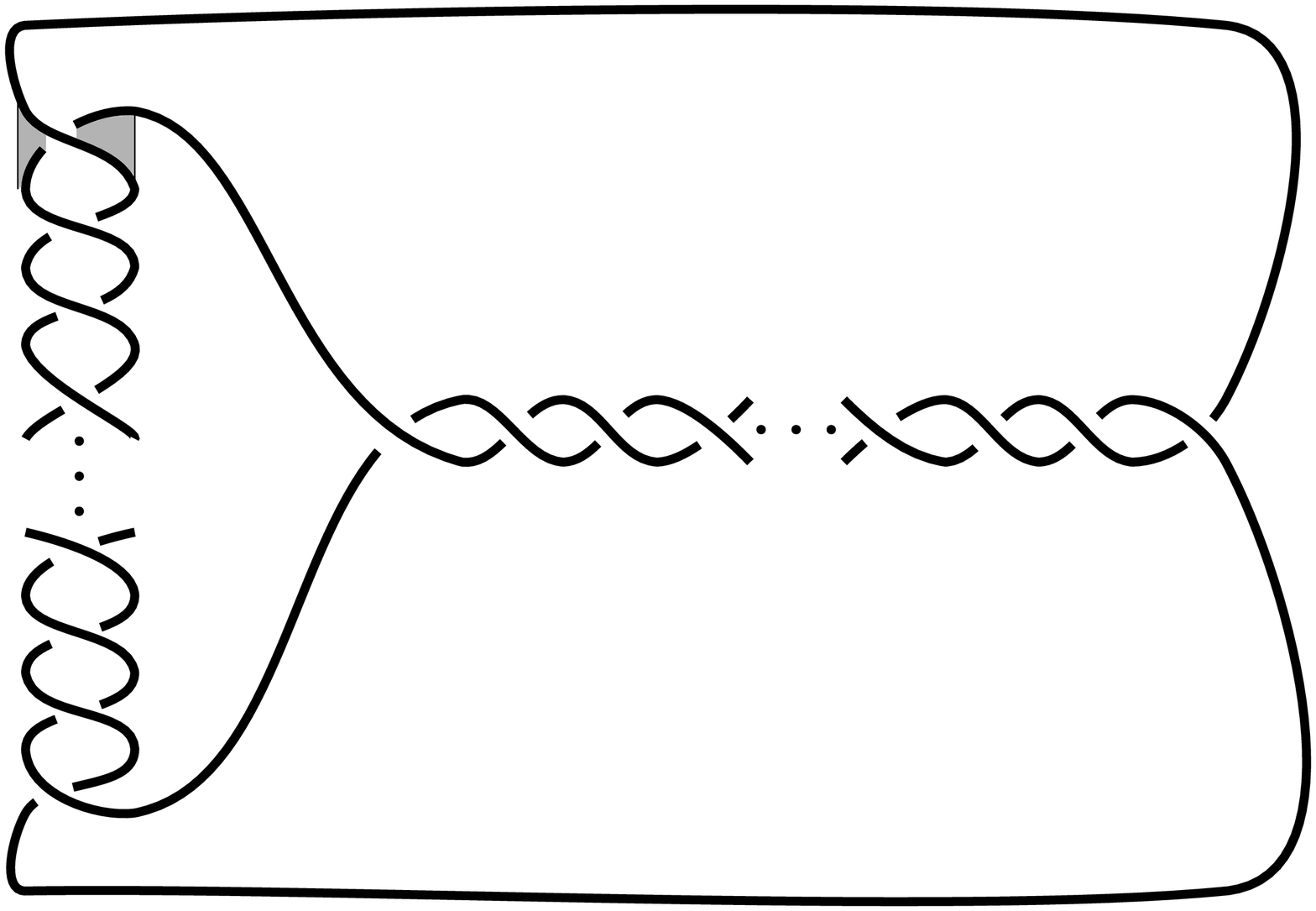}
		\caption{$d_2(K(m,n), K(m-1,n))=1$.}
		\label{fig:twist-operations}
	\end{subfigure}
	\qquad
	\begin{subfigure}[t]{.3\textwidth}
		\centering
		\includegraphics[width=1in]{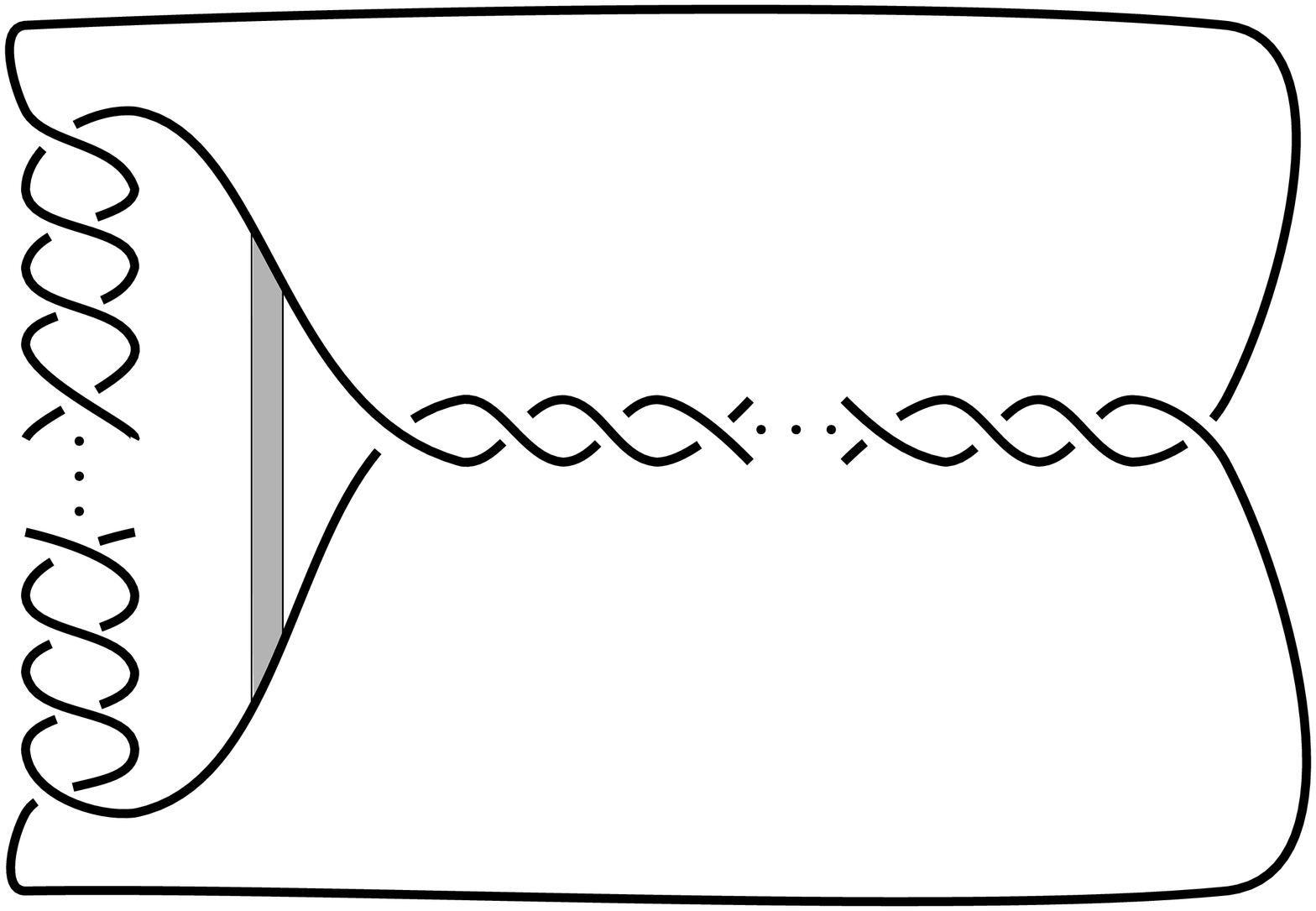}
		\caption{$d_2(K(m,n), T(2,m))=1$.}
		\label{fig:twist-operations}
	\end{subfigure}
\caption{Examples of H(2) moves relating generalized twist knots and torus knots.}\label{fig:twist-torus-operations}
\end{figure}

\begin{example}
The Jones polynomial of the $T(2, n)$ torus link is given (up to sign) by

\begin{equation}
\label{torus ex}
V_{T(2,n)}(t) =  t^{(n-1)/2} \left( \frac{ 1 + t + t^2 + (-1)^n t^{n+1}}{1+t} \right)  
= t^{(n-1)/2}  \left(t + 1 - t\sum_{k=0}^{n-1} {(-t)^{k}} \right).
\end{equation}
We verify this formula in Lemma \ref{torus jones} (see also \cite{Jones-note}). Here, we have assumed that when $n$ is even, the $T(2, n)$ torus link is oriented in the parallel manner, with both strands pointing up. From \eqref{torus ex}, it is immediately clear that $\det(T(2, n)) = |V_{T(2, n)}(-1)| = n$ and that the span is $n$. 
\end{example}

Another invariant we consider is related to tricolorability. The tricolorability of a link originates with Fox \cite{CF}; see also the definition in \cite{Adams}. The tricoloring number of $L$, $\tri(L)$, is the minimum number of proper, possibly trivial, tricolorings of a link. It can be obtained by evaluating the Jones polynomial at $t=e^{\pi i/3}$.
\begin{proposition}\cite[Theorem 1.13]{P2} 
The tricoloring number $\tri(L)$ is given by 
\begin{equation}
\label{tri equation}
	\tri(L) = 3|{V_{L}(e^{\frac{{\pi}i}{3}})^{2}}|. 
\end{equation}
\end{proposition}
In \cite{LM}, Lickorish and Millet previously showed that this evaluation of the Jones polynomimal is related to the dimension of the first homology group of the branched double cover $\Sigma(L)$ of $L$ with coefficients in $\Z_3$. Let $c(L)$ be the number of link components. Then,
\begin{equation}
\label{dim equation}
	V_{L}(e^{\frac{{\pi}i}{3}}) = \pm i^{c(L)-1}(i\sqrt{3})^{\dim H_1(\Sigma(L);\Z_3)}.
\end{equation}

\begin{example}
\label{tri trefoil}
The tricoloring number of the unknot is $\tri(U)=3$. For the trefoil knot $T(2, 3)$, the branched double cover is the lens space $L(3, 1)$, which has $H_1(\Sigma(L);\Z_3) =\Z_3$, and so $\dim H_1(\Sigma(L);\Z_3)=1$. By equations \eqref{tri equation} and \eqref{dim equation}, or by exhaustive enumeration, $\tri(T(2, 3)) = 3|(\pm (i\sqrt{3}))^2| = 9$.
\end{example}

The following properties of $\tri$ will also be useful:

\begin{lemma}\cite[Lemma 1.4, 1.5]{P2}
\label{tri properties}
\begin{enumerate}[label=(\roman*)]
	\item \label{power of three} $\tri(L) = 3^\beta$ for some $\beta\geq 1$.
	\item \label{connected sum} $\tri(L_1)\tri(L_2) = 3\tri(L_1\# L_2)$.
\end{enumerate}
\end{lemma}
Notice that Property \eqref{power of three} allows us to consider the positive integer-valued knot invariant $\beta$, which by definition is 
\begin{equation*}
	\beta(L) = \log_3(\tri(L)) = 1+ \log_3 |{V_{L}(e^{\frac{{\pi}i}{3}}})^2|.
\end{equation*} 
Notice that for a knot $K$, Eq. \eqref{dim equation} gives $\beta(L) = 1+ \dim  H_1(\Sigma(L);\Z_3)$. By the following result, the tricoloring number is known to give a lower bound on both the Gordian distance and the H(2)-Gordian distance. 
\begin{proposition}\cite[Theorem 5.5]{AK}, \cite[Proposition 4.2]{Miyazawa}
\label{bound}
	Let $K$ and $K'$ be a pair of knots related by a single H(2)-move or a crossing change. Then 
	\[
		\left| V_K(e^{\frac{{\pi}i}{3}}) / V_{K'}(e^{\frac{{\pi}i}{3}}) \right| \in \lbrace 1, \sqrt{3}, 1/\sqrt{3} \rbrace 
	\]
\end{proposition}
A similar statement is proven in \cite[Lemma 1.5]{P2}. We can repackage Proposition \ref{bound} as a lower bound on Gordian and H(2)-Gordian distance in terms of $\beta$ as follows. 
\begin{proposition}
\label{beta bound}
Let $K$ and $K'$ be a pair of knots. Then
\[
	|\beta(K)-\beta(K')| \leq d(K, K') \text{ and  }
	|\beta(K)-\beta(K')| \leq d_2(K, K'). 
\]
\end{proposition}
\begin{proof}
Suppose that $K$ and $K'$ are related by a single crossing change. Then Proposition \ref{bound} implies that 
	\[
		\left| V_K(e^{\frac{{\pi}i}{3}})^2 / V_{K'}(e^{\frac{{\pi}i}{3}})^2 \right| \in \lbrace 1, 3, 1/3 \rbrace,
	\]
and so $|\tri(K)/\tri(K')| \in \{1,3,1/3\}$. Applying log base 3, we equivalently have that $|\beta(K)-\beta(K')|\leq 1$. Thus for any pair of knots $K, K'$, $|\beta(K)-\beta(K')| \leq d(K, K')$. The proof for $d_2(K, K')$ is identical.
\end{proof}


\section{The Kauffman bracket and Jones polynomial}
\label{sec:bracket}

\begin{figure}
\centering
\begin{tikzpicture}
\node[anchor=south west,inner sep=0] at (0,0) {\includegraphics[width=4in]{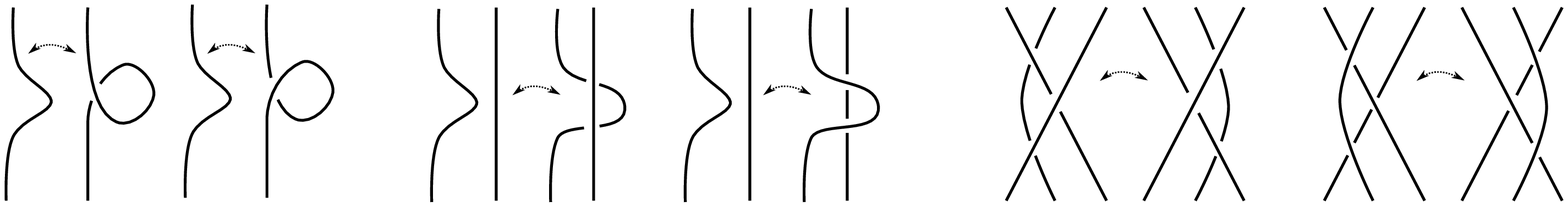}};
\node[label=above right:{Type I}] at (0, -0.6){};
\node[label=above right:{Type II}] at (3, -0.6){};
\node[label=above right:{Type III}] at (7, -0.6){};
\end{tikzpicture}
	\caption{Type I, Type II, and Type III Reidemeister moves.}
	\label{fig:reidemiester}
\end{figure}

In this section, we will review the definition of the Kauffman bracket \cite{Kauffman} and state some of its properties, most of which can be found in \cite{Kauffman} and \cite{EKT}. We then use the bracket to calculate the span of the Jones polynomial for several families of knots in a series of lemmas at the end of the section. 

The \emph{Kauffman bracket} $\langle L \rangle \in \Z[a, a^{-1}]$ of an unoriented link diagram $D_L$ is a Laurent polynomial defined by the following axioms:
\begin{enumerate}[label=(\roman*)]
	\item $\langle \bigcirc \sqcup D_L \rangle = \delta\langle D_L \rangle$, where $\delta =  (-a^2 -a^{-2})$,
	\item $\langle \backoverslash \rangle = a \hb + a^{-1} \vb$,
	\item $\langle \bigcirc \rangle = 1$.
\end{enumerate}
Recall that the bracket is not an invariant of links because it fails to be invariant under a Type I Reidemeister move (see Figure \ref{fig:reidemiester}). 
The deficiency of the bracket is corrected by a multiplicative factor that records the writhe $w(D_L)$ of the diagram. 
This yields the polynomial
\begin{equation*}
\label{normalization}
	X_L(a) = (-a)^{-3w(D_L)} \langle D_L \rangle
\end{equation*}
which is indeed a topological invariant of the link $L$. 
The polynomial $X_L(a)$ is equivalent to the Jones polynomial $V_L(t)$ after the change of variable $t = a^{-4}$ \cite{Jones, Kauffman}. 

\begin{figure}
\centering
\begin{tikzpicture}
\node[anchor=south west,inner sep=0] at (0,0) {\includegraphics[width=4in]{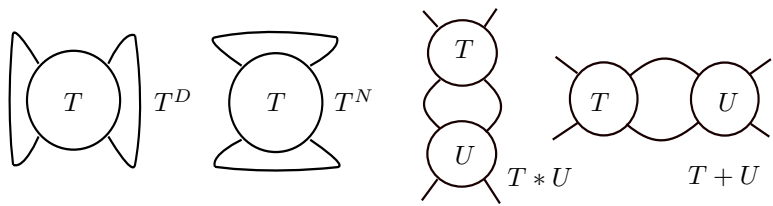}};
\node[label=above right:{$T^D$}] at (1.7, 1){};
\node[label=above right:{$T$}] at (0.5, 1){};
\node[label=above right:{$T^N$}] at (4.1, 1){};
\node[label=above right:{$T$}] at (3.2, 1){};
\node[label=above right:{$T*U$}] at (6.4, 0){};
\node[label=above right:{$T$}] at (5.7, 1.7){};
\node[label=above right:{$U$}] at (5.7, 0.3){};
\node[label=above right:{$T+U$}] at (8.8, 0){};
\node[label=above right:{$T$}] at (7.5, 1){};
\node[label=above right:{$U$}] at (9.2, 1){};
\end{tikzpicture}
	\caption{Denominator closure, numerator closure, vertical tangle product and horizontal tangle sum.}
	\label{tangle stuff}
\end{figure}

Let $T$ be a two-string tangle diagram. There are two standard closures of $T$, called the numerator and denominator closures $T^N$ and $T^D$. We denote the horizontal tangle sum by $T+U$, and the vertical tangle product by $T*U$. These operations are shown in Figure  \ref{tangle stuff}. The mirror of a tangle, obtained by changing all of the crossings, will be denoted $-T$. The zero tangle is $[0]=\hsmoothing$, and the infinity tangle is $[\infty]=\smoothing$. 

By applying the axioms (i) and (ii), we can write any tangle as an element in the bracket skein module of a disk with four marked points with basis $\{[0], [\infty] \}$, i.e. as a linear combination
\begin{equation*}
\label{kauffman poly}
	\langle T \rangle = f_T [0] + g_T [\infty],
\end{equation*}
where $f_T, g_T$ are polynomials in the ring $\Z[a, a^{-1}]$. The bracket vector $br(T)$ of the tangle $T$ is defined as $br(T)=[f_T, g_T]^T$, where here the superscript $T$ denotes the transpose. The following properties are well known and easy to verify (see for example, \cite{EKT}).
\begin{proposition}\cite[Proposition 2.2]{EKT}
\label{bracket properties}
\begin{enumerate}[label=(\roman*)]
	\item\label{bp 1} $\begin{bmatrix}\langle T^N \rangle \\ \langle T^D \rangle \end{bmatrix}  =\begin{bmatrix}\delta & 1 \\ 1 & \delta \end{bmatrix} br(T)$.
	\item\label{bp 2} $br(T+U) =  \begin{bmatrix} f_U & 0 \\ g_U & f_U + \delta g_U \end{bmatrix} br(T)$.
	\item\label{bp 3} $br(T*U) = \begin{bmatrix} \delta f_U + g_U & f_U \\ 0 & g_U \end{bmatrix} br(T)$.
\end{enumerate}
\end{proposition}
Given a tangle decomposition of a link diagram as $D_L=(T+U)^N$, Proposition \ref{bracket properties} implies that the Kauffman bracket of $D_L$ can be expressed as the evaluation of a bilinear form $(T, U) \mapsto \langle D_L \rangle$ in $\Z[a, a^{-1}]$, made explicit by 
\begin{equation}
\label{matrix}
	\langle D_L \rangle = br(T)^T \begin{bmatrix}\delta & 1 \\ 1 & \delta \end{bmatrix} br(U).
\end{equation}

Notice that because the span of a polynomial is preserved under multiplication by a monomial, the span of the Kauffman bracket $\langle D_L \rangle$ agrees with the span of $X_L(a)$. In particular, the span of the Kauffman bracket and the span of the Jones polynomial are both link invariants related by $\spn(\langle D_L \rangle) = 4 \spn (V_L).$ 
Kauffman \cite{Kauffman}, Murasugi \cite{Murasugi}, and Thistlethwaite \cite{Thistlethwaite} proved that for nonsplit alternating links, the span of the Jones polynomial equals the minimal crossing number of the link. 
We will make use of these properties to calculate the bracket polynomial for certain tangles and the span of the Jones polynomial for several classes of knots. The first calculation, giving the bracket vectors for tangles consisting of horizontal or vertical twists, as in Figure \ref{fig:twists}, is proved in \cite[Lemma 2.2]{KW}.

\begin{figure}
\centering
\begin{tikzpicture}
\node[anchor=south west,inner sep=0] at (0,0) {\includegraphics[width=3in]{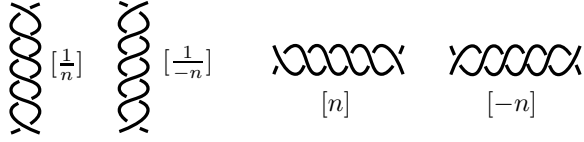}};
\node[label=above right:{$[\frac{1}{n}]$}] at (0.3, 0.5){};
\node[label=above right:{$[\frac{1}{-n}]$}] at (1.8, 0.5){};
\node[label=above right:{$[n]$}] at (3.9, 0){};
\node[label=above right:{$[-n]$}] at (6.1, 0){};
\end{tikzpicture}
\caption{Vertical and horizontal twist tangles are denoted $[1/n]$ and $[n]$, respectively.}
	\label{fig:twists}
\end{figure}

\begin{lemma}\cite[Lemma 2.2]{KW}
\label{twist bracket}
Let $n$ be a positive integer. Then
\begin{enumerate}[label=(\roman*)]
	\item \label{hor} $\langle [n] \rangle = a^n[0] + a^{n-2}\sum_{k=0}^{n-1}(-a^{-4})^k[\infty] = a^{n-2}\left( \frac{1-({-a^{-4}})^{n}}{1+a^{-4}} \right) [\infty] + a^{n}[0]$
	\item \label{vert} $\langle [1/n] \rangle = a^{-n+2}\sum_{k=0}^{n-1}(-a^{4})^k[0] + a^{-n}[\infty] = a^{-n}[\infty] + a^{-n+2} \left( \frac{1-(-a^{4})^{n}}{1+a^4} \right) [0]$
\end{enumerate}
\end{lemma}

Using the calculations for horizontal and vertical twists, we can now calculate the Jones polynomials for $T(2, n)$ torus links. This formula is well-known, but we include a calculation for completeness. 

\begin{lemma}
\label{torus jones}
Let $T(2, n)$ be a torus link for $n\neq0$ that is assumed to have parallel strand orientations when $n$ is even. Then the Jones polynomial of $T(2, n)$ is 
\begin{eqnarray*}
V_{T(2,n)}(t) &=& (-1)^{3n+1} t^{(n-1)/2} \left( \frac{ 1 + t + t^2 + (-1)^n t^{n+1}}{1+t} \right)  \\
&=&(-1)^{3n+1}t^{(n-1)/2}  \left(t + 1 - t\sum_{k=0}^{n-1} {(-t)^{k}} \right).
\end{eqnarray*}
\end{lemma}

\begin{proof}
By our conventions, the torus link $T(2, -n)$ with $n>0$ is obtained from the denominator closure of the tangle with a vertical $(n)$-twist, denoted $[1/n]$. 
By Lemma \ref{twist bracket}\eqref{vert} together with Lemma \ref{bracket properties}\eqref{bp 1}, 
\begin{equation}
\label{geometric series formula}
\langle {\left[ \frac{1}{n} \right] }^{D} \rangle = \delta{a^{-n}} + {a^{-n+2}}\left( \frac{1-(-a^4)^{n}}{1+a^{4}} \right).
\end{equation}
or alternatively,
\begin{equation}
\label{summation formula}
\langle {\left[ \frac{1}{n} \right] }^{D} \rangle  = \delta{a^{-n}} + a^{-n+2}\sum_{k=0}^{n-1} {{\left( -a^4 \right)}^k} .
\end{equation}
Let us first consider $\langle[1/n]^D\rangle$ written with a geometric series as in \eqref{geometric series formula}. Multiplying through by $\delta$ gives
\begin{equation*}
\langle {\left[ \frac{1}{n} \right] }^{D} \rangle  = -a^{-n+2} -a^{-n-2} + {a^{-n+2}}\left( \frac{1-(-a^4)^{n}}{1+a^{4}} \right).
\end{equation*}
Since we assumed that a torus link has parallel strand orientations when $n$ is even, we have that the writhe $w(T(2,-n)) = -n$, which gives us the polynomial
\[
X_{T(2,-n)}(a) = (-1)^{3n+1} a^{2n-2} \left( a^{4} + 1 - a^{4} \left( \frac{1- {\left( -a^4 \right)}^n}{1+a^4} \right)  \right).
\]
Next, we mirror the diagram of $T(2, -n)$ to obtain $T(2, n)$. Mirroring induces the change of variable $t\to t^{-1}$. Following this with a change of variables $a \to t^{-\frac{1}{4}}$, we obtain the Jones polynomial,
\[
V_{T(2,n)}(t) = (-1)^{3n+1} t^{(n-1)/2} \left( \frac{ 1 + t + t^2 + (-1)^n t^{n+1}}{1+t} \right).
\]
Stated for the case of a knot when $n$ is odd we have
\[
V_{T(2,n)}(t) = t^{(n-1)/2} \left( \frac{ 1 - t^3 - t^{n+1} +  t^{n+2}}{1-t^2} \right).
\]
Next we will consider the the version of the formula with the summation \eqref{summation formula}. This proceeds similarly to the first case and gives the polynomial
\[
X_{T(2,-n)}(a) = (-1)^{3n+1}a^{2n-2} \left(a^4 + 1 - a^4\sum_{k=0}^{n-1} {(-a^4)^{k}} \right).
\]
After mirroring and applying $t\to t^{-1}$, and the change of variables $a \to t^{-\frac{1}{4}}$, we obtain
\[
V_{T(2,-n)} = (-1)^{3n+1}t^{(n-1)/2}  \left(t + 1 - t\sum_{k=0}^{n-1} {(-t)^{k}} \right). \qedhere
\]
\end{proof}
For example, in the case of $n=3$, we obtain the polynomial $V_{T(2,3)}(t)=t+t^3-t^4$.

\begin{corollary}
\label{torus bracket}
When $n>1$ the span of the Jones polynomial of the $T(2, n)$ torus link is $n$.
\end{corollary}
\begin{proof}
This follows because for $n>1$, the $T(2, n)$ torus links are non-split and alternating. 
Note that when $n=0$, $T(2, 0)$ is the two-component unlink and when $n=1$, $T(2, 1)$ is the unknot, and these have spans 1 and 0, respectively.
\end{proof}

\begin{lemma}
\label{generalized twist bracket}
The span of the Jones polynomial of the generalized twist link $K(q, p)$ is $q+p$.
\end{lemma}
\begin{proof}
This follows because the generalized twist links $K(q, p)$ are non-split alternating links of minimal crossing number $q+p$.
\end{proof}

\begin{figure}
\centering
\begin{tikzpicture}
\node[anchor=south west,inner sep=0] at (0,0) {\includegraphics[width=0.6in]{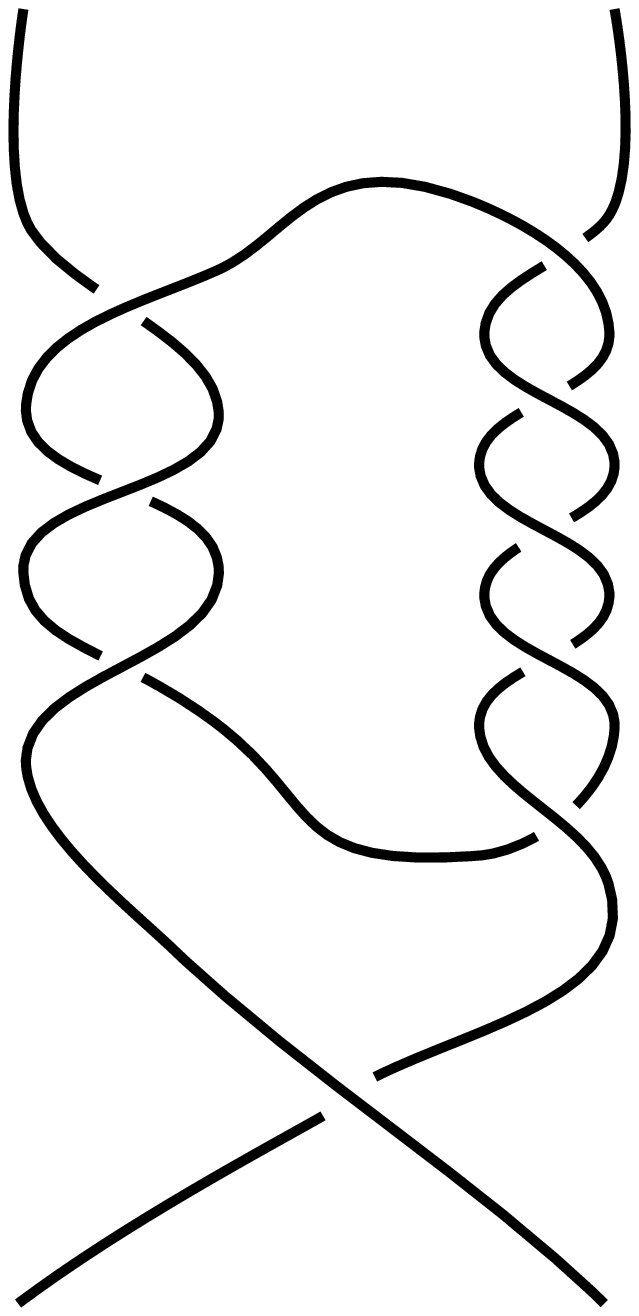}};
\node[label=above right:{$1/q$}] at (1.5, 1.5){};
\node[label=above right:{$-1/3$}] at (-1.2, 1.5){};
\node[label=above right:{$+1$}] at (-0.5, 0.3){};
\end{tikzpicture}
\caption{The tangle $T_q:=([-1/3]+[1/q])*[1]$. In this example, $q=5$. }
\label{fig:specialtangle}
\end{figure}

For several statements, we will need to calculate the entire Kauffman bracket vector of a special tangle $T_q:=([-1/3]+[1/q])*[1]$ that is shown in Figure \ref{fig:specialtangle}. 
\begin{lemma}
\label{bracket T}
The Kauffman bracket vector of $T_q:=([-1/3]+[1/q])*[1]$ is 
\[
	br (T_q ) = 
	\begin{cases}
	 	\begin{bmatrix} a^{-q-6} - 2a^{-q-2}-a^{-q+6} + a^{3q-2} \\
		a^{-q-8} - a^{-q-4} + a^{q-4}\sum_{k=1}^{q-1}(-a^{4})^k\end{bmatrix}, &q\geq3  \\ \\
	 	\begin{bmatrix} a^{2} - a^{-2} + a^{-6} \\
		- a^{-2} + a^{-6} \end{bmatrix}, &q=1 
	\end{cases}		
\]
\end{lemma}
\begin{proof}
By Lemma \ref{bracket properties}\eqref{bp 2}, when $q\geq3$, we calculate 
\begin{eqnarray*}
	br([-1/3]+[1/q]) &=& \begin{bmatrix} f_q & 0 \\ g_q & f_q + \delta g_q \end{bmatrix} 
	\begin{bmatrix} f_{-1/3} \\ g_{-1/3} \end{bmatrix} \\
	&=& \begin{bmatrix} a^{-q+2}\sum_{k=0}^{q-1}(-a^{4})^k & 0 \\  a^{-q}& a^{-q+2}\sum_{k=0}^{q-1}(-a^{4})^k + \delta a^{-q} \end{bmatrix} \begin{bmatrix} a-a^{-3}+a^{-7} \\ a^3 \end{bmatrix} \\
	&=&	\begin{bmatrix} (a^{-q+3} -a^{-q-1} + a^{-q-5})\sum_{k=0}^{q-1}(-a^4)^k \\ 
	a^{-q-7} - a^{-q-3} + a^{-q+5}\sum_{k=1}^{q-1}(-a^4)^k \end{bmatrix},
\end{eqnarray*}
where the indexing of the last sum takes into account cancellation of terms. Applying next Lemma \ref{bracket properties}\eqref{bp 3}, we obtain
\begin{eqnarray*}
	br(([-1/3]+[1/q])*[1]) &=& \begin{bmatrix} -a^3 & a \\ 0 & a^{-1} \end{bmatrix} 
	br(T_q) \\
	&=&	\begin{bmatrix} -a^3 & a \\ 0 & a^{-1} \end{bmatrix} 
	\begin{bmatrix} a^{-q-6} - 2a^{-q-2}-a^{-q+6} + a^{3q-2} \\ a^{-q-8} - a^{-q-4} + a^{q-4}\sum_{k=1}^{q-1}(-a^{4})^k\end{bmatrix} \\
	&=&\begin{bmatrix} a^{-q-6} - 2a^{-q-2}-a^{-q+6} + a^{3q-2} \\
	a^{-q-8} - a^{-q-4} + a^{q-4}\sum_{k=1}^{q-1}(-a^{4})^k\end{bmatrix},	
\end{eqnarray*}
where the first entry of the matrix product is simplified after cancelling numerous terms.

Next, consider the special case that $q=1$. By Lemma \ref{bracket properties}\eqref{bp 2}, we have
\begin{equation*}
	br([-1/3]+[1/q]) = \begin{bmatrix} a & 0 \\  a^{-1}& -a^{-3} \end{bmatrix} \begin{bmatrix} a-a^{-3}+a^{-7} \\ a^3 \end{bmatrix}  
	=\begin{bmatrix} a^2-a^{-2}+a^{-6} \\ -a^{-4}+a^{-8} \end{bmatrix}. \qedhere
\end{equation*}
\end{proof}

In the following two lemmas, we determine the span of the Kauffman bracket for two families of knots that are built from the special tangle $T_q$. The two classes of knots, $K_q$ and $K_{q,r}$, are shown in Figure \ref{fig:specialknots}. We will assume that $q\geq 1$ and $r\geq 3$ are both odd to ensure we have a knot. For example, taking $q=r=3$ gives the knot$K_{3, 3} = 10_{143}$.  Taking $q=5$ and $r=3$ gives $K_{5, 3} = 12n_{468}$, and taking $q=3$ and $r=5$ gives $K_{3, 5} = 12n_{570}$ \cite{knotinfo}. 

\begin{figure}[t]
	\centering
	\begin{subfigure}[t]{.5\textwidth}
		\centering
		    \begin{tikzpicture}
            \node[anchor=south west,inner sep=0] at (0,0) {\includegraphics[width=2.5in]{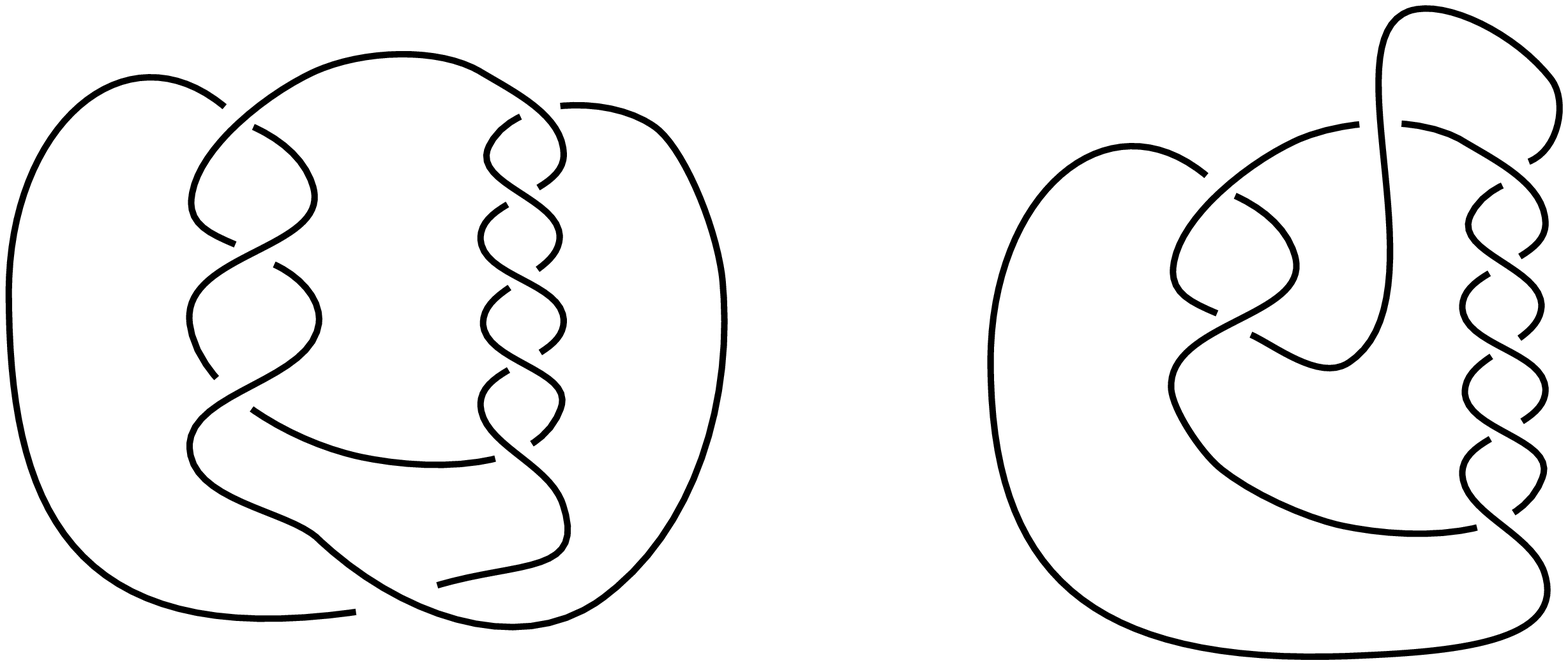}};
            \node[label=above right:{$q$}] at (2.2,1){};
            \end{tikzpicture}
		\caption{The knot $K_q:=T_q^D = (([-1/3]+[1/q])*[1])^D$}
		\label{fig:specialknot1}
	\end{subfigure}
	\qquad
	\begin{subfigure}[t]{.4\textwidth}
		\centering
			\begin{tikzpicture}
            \node[anchor=south west,inner sep=0] at (0,0) {\includegraphics[width=1.5in]{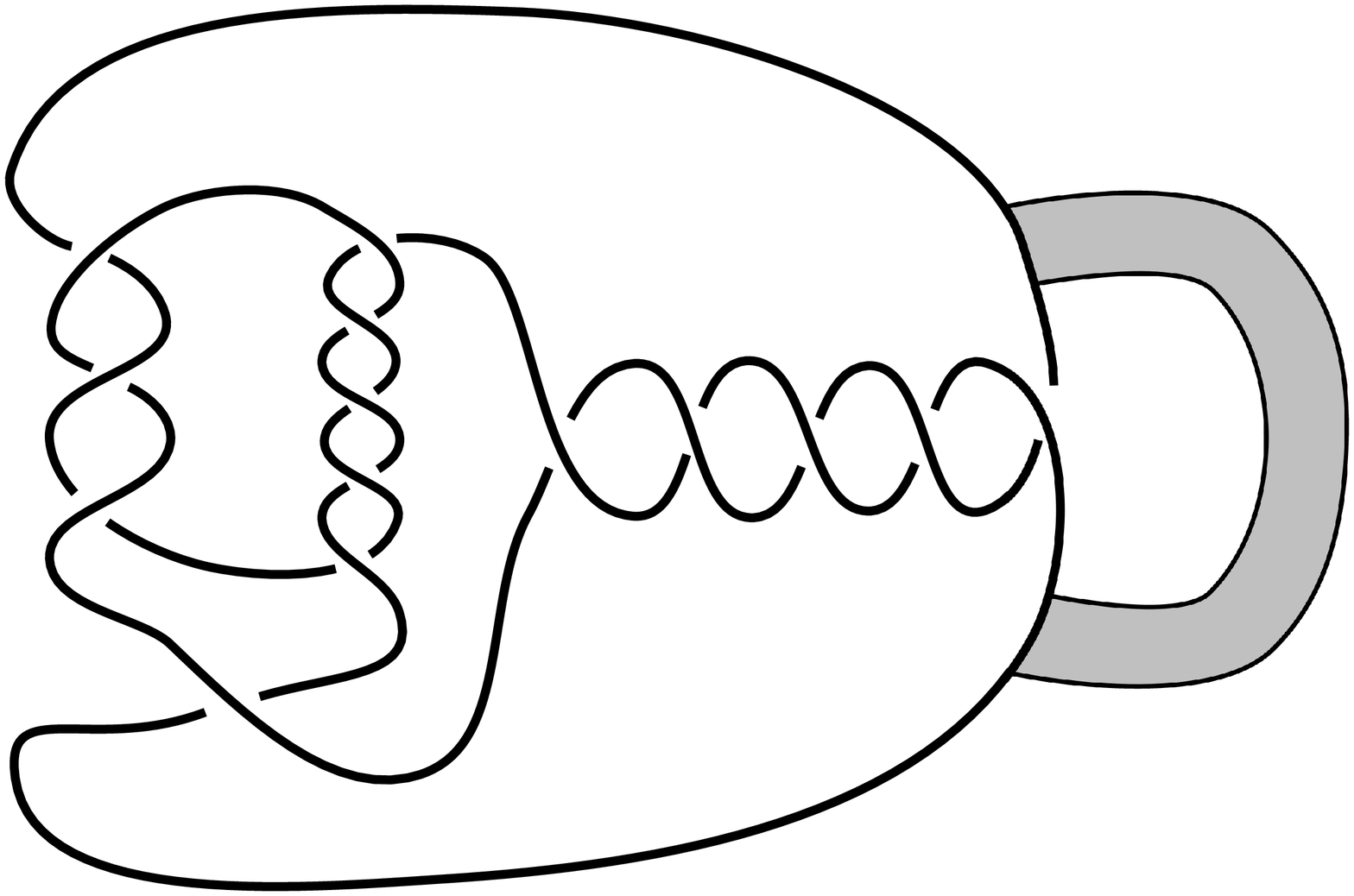}};
            \node[label=above right:{$q$}] at (1,1){};
            \node[label=above right:{$r$}] at (2,1.5){};
            \end{tikzpicture}
		\caption{The knot $K_{q,r}:=(T_q + [r])^N$}
		\label{fig:specialknot2}
	\end{subfigure}
	\caption{(Left) The knot $K_q:=T_q^D = (([-1/3]+[1/q])*[1])^D$. The isotopy demonstrates that $K_q$ is an alternating knot. (Right) The knot $K_{q,r}$. The shaded band indicates where an H(2)-move relates $K_{q,r}$ to $K_q$.}
	\label{fig:specialknots}
\end{figure}

\begin{lemma}
\label{knots Kq}
The span of the Jones polynomial of the knot $K_q$ is $q+3$.
\end{lemma}
\begin{proof}
The statement follows because for all $q$, $K_q$ is an alternating knot of $q+3$ crossings, as shown in Figure \ref{fig:specialknots}.
\end{proof}

\begin{lemma}
\label{knots Kqr}
The span of the Jones polynomial of the knot $K_{q,r}$ is $q+r+2$.
\end{lemma}
\begin{proof}
The family of knots $K_{q,r}$ is obtained as the numerator closure of the tangle sum $T_q+[r]$. Therefore Eq. \eqref{matrix} and Lemma \ref{bracket T} may be applied to calculate the bracket of $K_{q,r}$ as it appears in the diagram of $(T_q+[r])^N$. When $q\geq 3$, 
\begin{eqnarray*}
	\langle (T_q+[r])^N \rangle &=& br(T_q)^T \begin{bmatrix}\delta & 1 \\ 1 & \delta \end{bmatrix} br([r]) \\
	&=& \begin{bmatrix} a^{-q-6} - 2a^{-q-2}-a^{-q+6} + a^{3q-2} \\ a^{-q-8} - a^{-q-4} + a^{q-4}\sum_{k=1}^{q-1}(-a^{4})^k\end{bmatrix}^T 
	\begin{bmatrix}\delta & 1 \\ 1 & \delta \end{bmatrix}
	\begin{bmatrix} a^r \\ a^{r-2}\sum_{k=0}^{r-1}(-a^{-4})^k \end{bmatrix} \\
	&=& \begin{bmatrix} a^{-q-6} - 2a^{-q-2}-a^{-q+6} + a^{3q-2} \\ 
	a^{-q-8} - a^{-q-4} + a^{q-4}\sum_{k=1}^{q-1}(-a^{4})^k\end{bmatrix}^T 
	\begin{bmatrix} \delta a^r + a^{r-2}\sum_{k=0}^{r-1}(-a^{-4})^k \\ a^r+ \delta a^{r-2}\sum_{k=0}^{r-1}(-a^{-4})^k \end{bmatrix} 
\end{eqnarray*}
After noticing that the second entry of the last vector is a telescoping sum that reduces to $-a^{-3r}$, we obtain the product
\begin{equation}
\begin{aligned}
\label{long}
\langle (T_q+[r])^N \rangle &=& \left( a^{-q-6} - 2a^{-q-2}-a^{-q+6} + a^{3q-2} \right) \left( \delta a^r + a^{r-2}\sum_{k=0}^{r-1}(-a^{-4})^k \right) \\
 && + \left( a^{-q-8} - a^{-q-4} + a^{q-4}\sum_{k=1}^{q-1}(-a^{4})^k \right) \left( -a^{-3r} \right). 
\end{aligned}
\end{equation}
We need only extract the terms of highest and lowest exponents from Eq. \eqref{long}. The polynomial reduces to 
\[
	-a^{3q+r} + (\text{interior terms}) -a^{-3r-q-8},
\]
from which it is clear that the span of the bracket is $4(q+r+2)$ and the span of the Jones polynomial is $q+r+2$. 

When $q=1$,
\begin{eqnarray*}
	\langle (T_q+[r])^N \rangle &=& br(T_q)^T \begin{bmatrix}\delta & 1 \\ 1 & \delta \end{bmatrix} br([r]) \\
	&=& \begin{bmatrix} a^2 -a^{-2} + a^{-6} \\ -a^{-4} + a^{-8} \end{bmatrix}^T 
	\begin{bmatrix}\delta & 1 \\ 1 & \delta \end{bmatrix}
	\begin{bmatrix} a^r \\ a^{r-2}\sum_{k=0}^{r-1}(-a^{-4})^k \end{bmatrix} \\
	&=& \begin{bmatrix} a^2 -a^{-2} + a^{-6} \\ -a^{-4} + a^{-8} \end{bmatrix}^T 
	\begin{bmatrix} \delta a^r + a^{r-2}\sum_{k=0}^{r-1}(-a^{-4})^k \\ -a^{-3r} \end{bmatrix} \\
	&=& -a^{r+4} + (\text{interior terms}) -a^{-3r-6}.
\end{eqnarray*}
The span is then $4(r+3) = 4(r+q+2)$, and the span of the Jones polynomial is $q+r+2$. 
\end{proof}

Finally, we state a lemma that we will require when considering the spans of the Jones polynomials of knots.
\begin{lemma}
\label{span bound}
The span of the Jones polynomial of a knot cannot be one or two. 
\end{lemma}
\begin{proof}
We will show, equivalently, that there are no knots such that the normalized polynomial $X_K(a)$ has span $4$ or $8$. 
Take $K'$ to be the unknot, and let $K$ be any other knot. 
Ganzell showed that for any pair of knots $K, K'$, the difference of the polynomials $X_K(a) - X_{K'}(a)$ is divisible by $a^{12}-1$ \cite[Theorem 1]{ganzell}. 
This implies that $X_K(a)$ is of the form
\[
	X_K(a) = (c_M a^M + \cdots c_m a^m)(a^{12} - 1) + 1, 
\] 
which implies that the span of $X_K(a)$ is $\max \{12+M, 1\} - \min\{m, 1\}$. 
There are four cases. 
If $\spn X_K(a) = 12+M - m\geq 12$ or if $\spn X_K(a) = 0$, then clearly $\spn X_K(a)$ cannot be $4$ or $8$. 
In the case that  $\spn X_K(a) = 11+M$ is equal to $4$ or $8$, then $M = -7$ or $-3$, which contradicts that $M\geq m\geq 0$. 
In the case that $\spn X_K(a) = 1-m$ is equal to $4$ or $8$, then $m=-3$ or $-7$. Yet because $\max\{12+M, 1\}=1$, we have $M\leq -11$, which contradicts that $M\geq m\geq -7$. 
\end{proof}

\section{Hyperbolicity of quotient knot graphs}
\label{sec:hyperbolic}

\subsection{Hyperbolic graphs}
\label{sec:hyperbolic background}

Let $G$ be a graph with vertex set $V(G)$ and edge set $E(G)$. We only consider graphs that are connected, undirected, simple, and either finite or countably infinite. A connected graph $G$ can be endowed with a metric, making it into a metric space $(X_G, d)$ as follows (here, we follow the notation and conventions of \cite[Section 2.1]{JLM}). We first define a metric $d(v,w)$ on vertices $v, w \in V(G)$ as the minimum number of edges needed to connect $v$ to $w$, and implicitly assume that every edge in the graph has length 1. 
We then extend the distance to any pair of points $x,y$ on an edge $e\in E(G)$ by $d(x,y) = |x-y|$. 
The metric space $(X_G,d)$ can be considered a \emph{geodesic metric space}. 
In particular, the distance between any pair of points in $(X_, d)$ is realized by at least one rectifiable shortest path. 
Given three distinct points $x,y,z$ in $(X_G, d)$, a \emph{geodesic triangle} $\{[xy], [yz], [xz]\}$ is a triple of geodesics connecting the vertices of the triangle. 
For $\delta\geq 0$, we say that a geodesic triangle is $\delta$-thin if each edge is contained in a closed $\delta$-neighborhood of the union of the remaining two. 
The metric space $X_G$ is $\delta$-hyperbolic if every geodesic triangle in $X_G$ is $\delta$-thin, and we call $X_G$ Gromov hyperbolic (or just hyperbolic) if it is $\delta$-hyperbolic for some $\delta\geq0$. 
We will generally abuse notation and let $G$ (or $\mathcal{K}$) refer to both an abstract graph and a metric graph where the distance function is understood. 

Graphs provide many examples of hyperbolic metric spaces. 
\begin{example}
For example, any connected acyclic graph (i.e. a tree) is $\delta$-hyperbolic for all $\delta\geq0$. This is because in any geodesic triangle, one of the three edges will be equal to the union of the other two edges.
\end{example}

As another example, we observe the following:
\begin{lemma}
\label{diameter}
If $\diam(G)=d$, then $G$ is $ \frac{d}{2} $-hyperbolic.
\end{lemma}
\begin{proof}
Take a geodesic triangle with edges $[xy],[yz], [xz]$. Without loss of generality any point $p\in [xy]$ is contained in $N_{d/2}([xz]\cup[yz])$. 
\end{proof}

\subsection{Proofs of hyperbolicity for quotients of the Gordian and H(2)-Gordian graphs}

In order to prove Theorem \ref{Gamma beta}, we first observe that the indexing set will correspond with all natural numbers. 
\begin{lemma}
\label{betaDef}
The invariant $\beta$ takes all natural number values.
\end{lemma}
\begin{proof}
First notice that $\beta=1$ for the unknot. Next, let $\#^n T(2, 3)$ be the connected sum of $n$ trefoil knots. Since $\tri(T(2, 3))=9$, then by Lemma \ref{tri properties}\eqref{connected sum}, $\tri(\#^n T(2,3)) = 3^{n+1}$. So $\beta = n+1$ for $\#^n T(2, 3)$. 
\end{proof}
In section \ref{knot invariants} we discussed that the tricoloring invariant $\beta$ is related to the first homology of the branched double cover with $\mathbb{Z}/3\mathbb{Z}$-coefficients, and that these invariants give a lower bound on H(2)-Gordian distance. In particular, the H(2)-move and $\beta$ are compatible in the sense of Definition 1.2 of \cite{JLM}. The following statement can therefore be seen as a corollary of Theorem 5.1 of \cite{JLM}. However, we give a direct argument. 
\begin{theorem}
\label{Gamma beta}
The quotient graphs $\mathcal{QK}_{\backoverslash}^{\beta}$ and $\mathcal{QK}_{\smoothing}^{\beta}$ are $\delta$-hyperbolic for all $\delta$.
\end{theorem}
\begin{proof}
We will prove that $\mathcal{QK}_{\backoverslash}^{\beta}$ is an infinite path graph $P_{\mathbb{N}}$. From Lemma \ref{betaDef}, we see that the invariant $\beta$ realizes all natural numbers, and so the vertex set of $\mathcal{QK}_{\backoverslash}^{\beta}$ is $\mathbb{N}$. Using the same family of connected sums of trefoils and the fact that the trefoil is related to the unknot by a crossing change, then there exists an edge between $\#^n T(2, 3)$ and $\#^{n+1}T(2, 3)$ for all $n \ge 1$.  By Proposition \ref{beta bound}, we also know that a crossing change cannot change $\beta$ by more than one, therefore no other edges can exist in this graph. Thus $\mathcal{QK}_{\backoverslash}^{\beta}$   is an infinite path graph  $P_{\mathbb{N}}$. Since  $P_{\mathbb{N}}$ is connected and acyclic, $\mathcal{QK}_{\backoverslash}^{\beta}$  is $\delta$-hyperbolic for all $\delta$. 

The proof for $\mathcal{QK}_{\smoothing}^{\beta}$ is identical after noticing that the trefoil knot is also unknotted with a single H(2)-move. 
\end{proof}

\begin{theorem}
\label{Gamma det}
	The quotient graphs $\mathcal{QK}_{\backoverslash}^{\det}$ and  $\mathcal{QK}_{\smoothing}^{\det}$ are $\delta$-hyperbolic for all $\delta\geq1$.
\end{theorem}
\begin{proof}
We will first show that the quotient graph $\mathcal{QK}_{\backoverslash}^{\det}$ has $\diam(\mathcal{QK}_{\backoverslash}^{\det}) \le 2$. 
Note that the equivalence classes span the odd natural numbers, and that there exists a representative twist knot $K(n,2)$ for each equivalence class $[2n+1]$ for $n \geq 1$.
Then every equivalence class $[2n+1]$ is adjacent to $[1]$, because twist knots $K(n, 2)$ have unknotting number one.
Therefore $\text{diam}(\mathcal{QK}_{\backoverslash}^{\det}) \le 2$. By Lemma \ref{diameter} it is $\delta$-hyperbolic for all $\delta\geq1$.

For $\mathcal{QK}_{\smoothing}^{\det}$, we will similarly show that the quotient graph has diameter less than or equal two.  
Here, there exists a representative torus knot $T(2,2n+1)$ for each equivalence class $[2n+1]$ for $n \geq 1$.
Again, every equivalence class $[2n+1]$ is adjacent to $[1]$, because $T(2, 2n+1)$ torus knots have H(2)-unknotting number one. 
Therefore $\text{diam}(\mathcal{QK}_{\backoverslash}^{\det}) \le 2$. Again by Lemma \ref{diameter}, it is $\delta$-hyperbolic for all $\delta\geq1$.
\end{proof}

Recall from Lemma \ref{span bound} that there are no knots whose Jones polynomials have span $1$ or span $2$. Thus we define $\widetilde{\mathcal{QK}}_{\backoverslash}^{\spn}$ to be $\mathcal{QK}_{\backoverslash}^{\text{span}} - \{ [1],[2] \}$, i.e. the quotient Gordian knot graph under span with vertex set indexed by the integers $\{0, 3, 4, 5, \dots \}$. 
\begin{theorem}
\label{Gamma span}
	The quotient graph $\widetilde{\mathcal{QK}}_{\backoverslash}^{\spn}$ is $\delta$-hyperbolic for all $\delta\geq1$.
\end{theorem}
\begin{proof}
We claim that $\diam(\widetilde{\mathcal{QK}}_{\backoverslash}^{\spn}) \leq 2$. This follows because every equivalence class $[n]$, for $n\geq 3$, is adjacent to $[0]$. This adjacency is realized by the twist knots $K(m,2)$, which have span $m+2$. The result then follows from Lemma \ref{diameter}. 
\end{proof}

We similarly define $\widetilde{\mathcal{QK}}_{\smoothing}^{\text{span}}$ to be $\mathcal{QK}_{\smoothing}^{\text{span}} - \{ [1],[2] \}$, understood to have vertex set index by $\{0, 3, 4, 5, \dots \}$. 
\begin{theorem}
\label{Gamma 2 span}
	The quotient graph $\widetilde{\mathcal{QK}}_{\smoothing}^{\spn}$ is $\delta$-hyperbolic for all $\delta\geq1$.
\end{theorem}
\begin{proof}
We claim that $\diam(\widetilde{\mathcal{QK}}_{\smoothing}^{\spn}) \leq 2$. To see this, first observe that every equivalence class $[2n+1]$, for odd $n\geq 1$, is adjacent to $[0]$. This adjacency is realized by the $T(2, 2n+1)$ torus knots, which have span $2n+1$.  Additionally, for every $n\geq 3$, the class $[n]$ is adjacent to $[n+1]$. This adjacency is realized by the twist knots $K(n, 1)$ and $K(n+1, 2)$, which are related by the resolution of a single crossing. The result then follows from Lemma \ref{diameter}.
\end{proof}

\begin{remark}
It is possible to construct distinct pairs of knots related by a crossing change that share the invariants considered here. For example, this could be done with pretzel knots of the form $P(q, -1, -q)$ for $q\geq 3$ odd. The knot $P(q, -1, -q)$ is related to $P(q, 1, -q)$ by a crossing change. Because $P(q, 1, -q)\simeq P(-q, 1, q)$ is the mirror of $P(q, -1, -q)$, their Jones polynomials have the same span, and the values of $\beta$ and $\det$ will agree. However, here we consider only simple knot graphs (ignoring self edges). 
\end{remark}

\section{Graph isomorphism type}
\label{sec:graphs}

In this section we study the possible isomorphism types of the graphs $\mathcal{QK}_{\backoverslash}^p$ and $\mathcal{QK}_{\smoothing}^p$ for $p=\beta, \det$ and $\spn$. In the case of the invariant $\beta$, the isomorphism types are easily identified.

\begin{theorem}
Each of $\mathcal{QK}_{\backoverslash}^{\beta}$ and $\mathcal{QK}_{\smoothing}^{\beta}$ are isomorphic to the path graph $P_{\mathbb{N}}$. 
\end{theorem}

\begin{proof} 
It follows immediately from the proof of Theorem \ref{Gamma beta} that the vertex sets $V(\mathcal{QK}_{\backoverslash}^{\beta}) = \mathbb{N}$ and $V(\mathcal{QK}_{\smoothing}^{\beta}) = \mathbb{N}$, and that the edge sets $E(\mathcal{QK}_{\backoverslash}^{\beta}) = \{ [n], [n+1] \}$ and $E(\mathcal{QK}_{\smoothing}^{\beta}) = \{ [n], [n+1] \}$ for all $n\in\mathbb{N}$.  
\end{proof}
Unlike the quotient knot graphs for $\beta$, the graphs $\mathcal{QK}_{\backoverslash}^{\det}$ and $\mathcal{QK}_{\smoothing}^{\det}$ have finite diameter. We can exploit the fact that the determinant is multiplicative over connected sums in order to obtain the following two statements. By $\mathbb{K}_{\infty}$, we mean the complete graph on countably many vertices. 

\begin{theorem}\label{complete subgraph}
In both $\mathcal{QK}_{\backoverslash}^{\det}$ and $\mathcal{QK}_{\smoothing}^{\det}$ , the subgraph induced by the set of vertices $\{[n^k], k\geq 0\}$ is isomorphic to $\mathbb{K}_{\infty}$ for all odd $n\geq3$.
\end{theorem}	
	
\begin{proof}
Let $n$ be positive and odd. 
In $\mathcal{QK}_{\backoverslash}^{\det}$, any vertex $[n^i]$ may be realized by the twist knot $K(2, \frac{n^i-1}{2})$ because $\det(K(2,\frac{n^i-1}{2}))=n^i$. 
Any such twist knot is related to the unknot, which has $\det(U)=1$, establishing the edges $\{[n^0], [n^i]\}$. 
Next for $n\geq 3$, consider any two vertices $[n^i], [n^j]$, with $i<j$, in the set $\{[n^k], k\geq 1\}$. 
The vertex $[n^j]$ is realized by a connected sum of twist knots, with
\[
	\det (K(2, \frac{n^i-1}{2}) \# K(2, \frac{n^{j-i}-1}{2}) )= n^i \cdot n^{j-i} = n^j.
\]
The connected sum $K(2, \frac{n^i-1}{2}) \# K(2, \frac{n^{j-i}-1}{2}) $ is related to $K(2, \frac{n^i-1}{2})$ by a single crossing change in the clasp of $K(2, \frac{n^{j-i}-1}{2})$ which unknots that summand. 
This gives the edge $\{[n^i], [n^j]\}$. 
Thus all vertices $\{ [n^k]; n\geq 3, k\geq 0 \}$ are adjacent, and so the induced subgraph is isomorphic to $\mathbb{K}_{\infty}$. 

In $\mathcal{QK}_{\smoothing}^{\det}$, the proof is similar, except that we use $T(2, 2n+1)$ torus knots rather than twist knots. 
Since $\det(T(2, 2n+1)) = 2n+1$ and $T(2, 2n+1)$ is related to the unknot by a single H(2)-move, this gives the edges $\{[n^0], [n^i]\}$. 
For the edges $\{[n^i], [n^j]\}$, with $n\geq3, j>i\geq1$, we take the connected sum 
\[
	T(2, \frac{ n^i-1}{2} ) \# T(2, \frac{n^{j-i}-1}{2}), 
\]
which has determinant $n^j$, and is related to $T(2, \frac{n^i-1}{2})$ of determinant $n^i$ by the single H(2)-move which unknots the second summand.
\end{proof}

\begin{example}
\label{edges with det}
It is certainly possible to construct additional edges in the graphs $\mathcal{QK}_{\backoverslash}^{\det}$ and $\mathcal{QK}_{\smoothing}^{\det}$, and we describe here three examples. 
\begin{enumerate}[label=(\roman*)]
\item By an argument similar to that of Theorem \ref{complete subgraph}, there exist edges in $\mathcal{QK}_{\backoverslash}^{\det}$ and $\mathcal{QK}_{\smoothing}^{\det}$ for any pair of vertices $[n], [m]$, with $n,m$ odd and where $m$ divides $n$. In this case, let $n = dm$. Then for $\mathcal{QK}_{\backoverslash}^{\det}$, the relevant crossing change is found in the second summand of
\[
K(2, \frac{ m-1}{2} ) \# K(2, \frac{d-1}{2} ),
\]
and for $\mathcal{QK}_{\smoothing}^{\det}$ the relevant H(2)-move is found in the second summand 
\[
T(2, \frac{ m-1}{2} ) \# T(2, \frac{d-1}{2} ).
\]

\item In  $\mathcal{QK}_{\backoverslash}^{\det}$, the edges $\{[2n+1],[2n+5]\}$ for $n\geq 1$ are realized by the pair of twist knots $K(2,n)$, $K(2, n+2)$. In the graph $\mathcal{QK}_{\smoothing}^{\det}$, the edges $\{[2n+1],[2n+3]\}$ for $n\geq 1$ are again realized by pairs of twist knots $K(2,n)$, $K(2, n+1)$.

\item Consider the $P(p,q,r)$ pretzel knots, where in order to obtain a knot we assume that at most one of $p,q,r$ is even. It is well known that the $P(p,q,r)$ pretzel knots have determinant $|pq+pr+qr|$. It is also easy to see that the knot $P(p,q,r)$ is related to $P(p,q-2,r)$ via a crossing change and is related to $P(p,q-1,r)$ via an H(2)-move. In particular, for any $p\geq1$, we have that $\det(P(p,3,1)) = 4p+3$ and $\det(P(p,2,1)) = 3p+2$ and $\det(P(p,1,1)) = 2p+1$. Hence, there exist edges in $\mathcal{QK}_{\backoverslash}^{\det}$ relating $2p+1$ and $4p+3$ for all $p\geq 1$ and there exist edges in $\mathcal{QK}_{\smoothing}^{\det}$ relating $2p+1$ and $3p+2$ for all $p\geq 1$. 
\end{enumerate}
\end{example}
One may work with other explicit examples to construct additional edges in these graphs. However, we know of no universal construction that produces a pair of knots realizing an arbitrary edge $\{[2n+1], [2m+1]\}$ for any $n, m$. Therefore, we ask:
\begin{question}
Are the graphs $\mathcal{QK}_{\backoverslash}^{\det}$ and $\mathcal{QK}_{\smoothing}^{\det}$ isomorphic to $\mathbb{K}_{\infty}$?
\end{question}

Next we turn our focus to the graphs $\widetilde{\mathcal{QK}}_{\backoverslash}^{\spn}$ and $\widetilde{\mathcal{QK}}_{\smoothing}^{\spn}$. 
Recall from Lemma \ref{span bound} that for knots the invariant $\spn$ takes values in $\{0,3,4,\dots\}$, therefore we remove $n=1,2$ from the indexing sets of $\widetilde{\mathcal{QK}}_{\backoverslash}^{\spn}$ and $\widetilde{\mathcal{QK}}_{\smoothing}^{\spn}$. 

\begin{proposition}
\label{Gamma span structure}
The edges $\{[0], [n]\}$ are in $E(\widetilde{\mathcal{QK}}_{\backoverslash}^{\spn})$ for all $n\geq 3$ and the edges 
$\{[m], [n]\}$ are in $E(\widetilde{\mathcal{QK}}_{\backoverslash}^{\spn})$ for all $m,n$, where $n-2\geq m\geq 3$. 
\end{proposition}

\begin{proof}
The twist knots $K(2, n-2)$ have span $n$ for all $n\geq 3$. Because these all have unknotting number one, they realize the edges $\{[0],[n]\}$. Next, consider an arbitrary pair $[m],[n]$, where $m\geq 3$ and $n\geq m+2$. If $n=m+2$, then the twist knots $K(2, m)$ and $K(2, m-2)$ are related by a crossing change and have spans $n$ and $m$, respectively. If $n > m+2$, then $K(2, m-2)$ and the connected sum $K(2, m-2)\#K(2, n-m-2)$ have spans $m$ and $n$, respectively, because span is additive over connected sum. An edge relates this pair because there is a crossing change that unknots the second summand. These knots realize the edges $\{[m],[n]\}$ when $m\geq 3$ and $n\geq m+2$.
\end{proof}

Notice that Proposition \ref{Gamma span structure} nearly shows that the graph $\widetilde{\mathcal{QK}}_{\backoverslash}^{\spn}$ is isomorphic with $\mathbb{K}_{\infty}$. The only missing edges in the proof are those of the form $\{ [n], [n+1]\}$. We are in fact aware of some sporadic pairs of knots related by a crossing change whose spans differ by one (see for example \cite{Moon}.) These pairs are listed below and shown in Figure \ref{fig:span1}. 

\begin{itemize}
	\item The twist knot $K(2,3)=5_2$ (span 5) and the pretzel knot $P(3,-3,2)=8_{20}$ (span 6). 

	\item The pretzel $P(3,3,-2) =T(3,4)=8_{19}$ (span 5) and the pretzel knot $P(5,3,-2)$ (span 6).

	\item The pretzel $P(5,3,-2) = T(3,5) = 10_{124}$ (span 6) and $T(2, 7) =7_1$ (span 7).

	\item The knot $7_3$ (span 7) and the pretzel knot $P(-5,3,2) = 10_{126}$ (span 8).

\end{itemize}

\begin{figure}
	\includegraphics[width=4in]{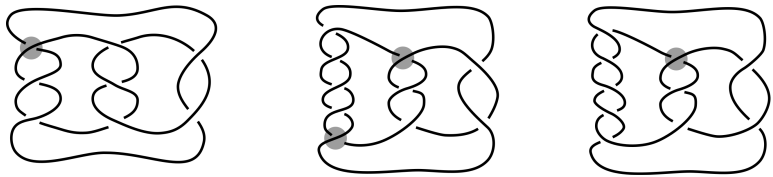}
	\caption{Each grey disk indicates where a crossing change relates the knot pictured to another knot whose span of the Jones polynomial differs by one.}
	\label{fig:span1}
\end{figure}

However, we know of no general construction of knots related by a crossing change that alters the span by one. Therefore we ask:
\begin{question}
Is there an edge between $[n]$ and $[n+1]$ for all $n \geq 3$ in $\widetilde{\mathcal{QK}}_{\backoverslash}^{\spn}$? \end{question}

For the following statement, we will consider a quotient of the H(2)-Gordian graph of \emph{links}, denoted $\mathcal{L}_{\smoothing}$, rather than the version for knots. In particular, the vertex set of $\mathcal{L}_{\smoothing}$ consists of isotopy classes of unoriented links. An H(2)-move may or may not be component preserving, and an edge in this graph exists whenever a pair of links $L, L'$ is related by any H(2)-move. 
We remark that while the Jones polynomial of links is sensitive to orientation, and the H(2)-move may or may not respect strand orientation, the span of the Jones polynomial (or Kauffman bracket) is insensitive to orientation. 
Thus the quotient graph $\mathcal{QL}^{\spn}_{\smoothing}$ is well-defined for unoriented links.

\begin{theorem}\label{Gamma2 span structure}
The graph $\mathcal{QL}^{\spn}_{\smoothing}$ is isomorphic to $\mathbb{K}_{\infty}$. 
\end{theorem}

\begin{proof}
We take our indexing set for the vertices of $\mathbb{K}_{\infty}$ to be the set of non-negative integers. For every pair $n, m$, we will construct an edge in $\mathcal{QL}^{\spn}_{\smoothing}$. There are seven cases in our analysis; the first three are specific to the vertices $[0],[1],[2]$, and the latter cases are more general. We will use knots to construct edges whenever possible (see Remark \ref{knot argument} below), even if a simpler construction can be made with links.

\textbf{Case 1.} Recall that span $T(2,0)=1$, span $T(2,1)=0$, and span $T(2,n)=n$ for all $n\geq 2$. Edges incident to $[0]$ are realized by an H(2)-move that unknots the $T(2, n)$ torus knot or link, as seen in Figure \ref{fig:twist-torus-operations}(a). 

\textbf{Case 2.} Consider edges incident to $[1]$. The edge $\{[1],[2]\}$ is realized by an H(2)-move relating the two-component and three-component unlinks. If $n>2$, the edge $\{[1],[n]\}$ is found by taking the disjoint union $U\sqcup T(2, n-1)$. By axiom (i) of the definition of the bracket, this has span $n$, and is related to the two-component unlink (of span 1) by the banding that unties the torus link. See Figure \ref{fig:linksum}.

\textbf{Case 3.} Consider edges incident with $[2]$. The edge $\{[2],[3]\}$ is realized by the Hopf link and trefoil knot, and the edges $\{[2], [n]\}$ for $n\geq 4$ are realized by the H(2)-move that unknots the second summand of $T(2, 2)\# T(2, n-2)$. (Figure \ref{fig:linksum} defines this connect sum.)

\textbf{Case 4.} Consider $\{[n], [n+1]\}$ for $n\geq 2$. These edges are realized by torus knots and links because $T(2, n)$ is related to $T(2, n+1)$ by an H(2)-move. Alternatively, these edges can be realized by pairs of twist knots $K(2, n-2)$ and $K(2, n-1)$. Any twist knot $K(2, n-2)$ has span $n$, and these pairs are related by a single H(2)-move in the twist region, as in Figure \ref{fig:twist-torus-operations}(b). 

\textbf{Case 5.}  Consider $\{[m], [n]\}$ when $n-m>1$ and $n > m \geq 4$ are even. 
This edge is realized by the knots $K_q$ and $K_{q,r}$ where $q=m-3$ and $r=n-m+1$. 
Here, $m\geq 4$ even implies $q \geq 1$ odd and $n\geq 6$ even implies $r\geq 3$ odd. 
These knots are shown in Figure \ref{fig:specialknots}. By Lemma \ref{knots Kq}, the span of $K_q$ is $q+3 = m$ and by Lemma \ref{knots Kqr}, the span of $K_{q,r}$ is $q+r+2 = n$. These knots are related by the shaded band shown in Figure \ref{fig:specialknots}. 

\textbf{Case 6.}  Consider $\{[m], [n]\}$ when $n-m>1$ and $n>m\geq 3$ are odd. Because $n-m$ is even, the generalized twist link $K(m, n-m)$ is a knot. An H(2)-move relates the torus knot $T(2, m)$, which has span $m$, to $K(m, n-m)$, which has span $n$. This is shown in Figure \ref{fig:twist-torus-operations}(c). 

\textbf{Case 7.}  Consider $\{[m], [n]\}$ when $n-m>1$ and $n>m\geq 3$ are of different parity. Then $n-m$ is odd, and so $T(2, n-m)$ is a knot. The connected sum $T(2, n-m) \# K(2, m-2)$ has span $n$, and is related to $K(2, m-2)$, which has span $m$, by the H(2)-move that unknots the torus knot summand. 

This case analysis establishes that the pair $\{[n],[m]\}\in E(\mathcal{QL}^{\spn}_{\smoothing})$ for all $n, m\in \mathbb{N}\cup\{0\}$.
\end{proof}

\begin{remark}
\label{knot argument}
Notice that most of the edges constructed in the proof above are also contained in the graph $\widetilde{\mathcal{QK}}^{\spn}_{\smoothing}$. The only exception are the edges from $[0]$ to $[2n]$, realized by the torus links $T(2, 2n)$, and the edges that are incident with the vertices $[1]$ and $[2]$. This prompts us to ask:

\end{remark}

\begin{question}
Are there edges $\{[0],[2n]\}$, for $n\geq 2$, realized by pairs of knots? Is the graph $\widetilde{\mathcal{QK}}^{\spn}_{\smoothing}$ isomorphic to $\mathbb{K}_{\infty}$?
\end{question}

\begin{figure}
	\centering
	\includegraphics[width=4in]{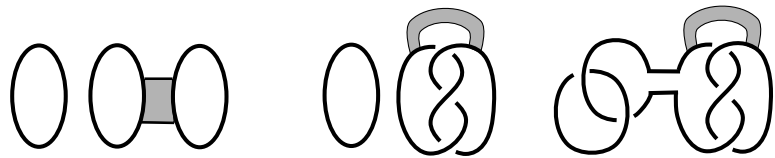}
	\caption{The connected sums and bandings of links in case (2) and case (3) of Theorem \ref{Gamma2 span structure}.}
	\label{fig:linksum}
\end{figure}

\subsection*{Acknowledgements}
The authors appreciate the support of the Honors Summer Undergraduate Research Program (HSURP) at Virginia Commonwealth University, and thank Slaven Jabuka and Beibei Liu for helpful conversations.

\bibliographystyle{alpha}
\bibliography{bibliography}

\end{document}